\documentclass[12pt,reqno]{amsart}
\usepackage{amsmath,amssymb,extarrows}
\usepackage{hyperref}
\usepackage{url}
\usepackage{tikz,enumerate}
\usepackage{diagbox}
\usepackage{appendix}
\usepackage{epic}
\usepackage{comment}
\usepackage{bm}
\usepackage{adjustbox}
\usepackage{booktabs}
\usepackage{array}

\usepackage{float} 

\usepackage{cite}
\usepackage{hyperref}
\usepackage{array}

\allowdisplaybreaks[4]

\newtheorem{theorem}{Theorem}
\newtheorem{corollary}[theorem]{Corollary}

\newtheorem{lemma}[theorem]{Lemma}

\theoremstyle{definition}

\theoremstyle{remark}
\newtheorem{rem}{Remark}
\numberwithin{equation}{section}
\numberwithin{theorem}{section}
\numberwithin{defn}{section}

\newcommand{\e}{\mathrm e}
\newcommand{\ii}{\mathrm{i}}

\begin{document}
\title[Proofs of Some Kanade--Russell Mod 12 Conjectures]
 {Proofs of Some Kanade--Russell Mod 12 Conjectures}

\author{Liuquan Wang}
\address{School of Mathematics and Statistics, Wuhan University, Wuhan 430072, Hubei, People's Republic of China}
\email{wanglq@whu.edu.cn;mathlqwang@163.com}



\date{September 8, 2026}
\dedicatory{Dedicated to Professor Heng Huat Chan on the occasion of his forthcoming 60th birthday}

\begin{abstract}
Kanade and Russell conjectured seventeen Rogers--Ramanujan type identities of modulus 12.  Eleven of these identities involving triple sums were proved by Bringmann--Jennings-Shaffer--Mahlburg and by Rosengren. Motivated by these works and using similar methods, we settle all the six remaining conjectures including two triple sum identities labeled $I_{5a}$ and $I_{6a}$ originated from Russell's thesis and four quadruple sum identities labeled 7, 7a, 8 and 8a. Our proof of the triple sum identities combines linear recurrences, $q$-difference equations, and $q$-series summation formulas. For the quadruple sum identities, we represent the sums as contour integrals whose integrands are infinite products and evaluate them by residue calculus. The resulting residues reduce to single sum cubic basic hypergeometric series, and we are able to express them as infinite products.
\end{abstract}

\maketitle

\section{Introduction}\label{sec-intro}
Rogers--Ramanujan type identities are sum-to-product identities that express $q$-hypergeometric series as elegant infinite products. They constitute a central theme in the theory $q$-series and have prominent applications in combinatorics \cite{MacMahon}, number theory \cite{GOW}, the representation theory of Lie algebras \cite{Lepowsky-Wilson1982}, and mathematical physics \cite{Baxter}. Their study dates back to Rogers' discovery \cite{Rogers1894}, around 1894, of the following celebrated identities:
\begin{align}\label{RR}
\sum_{n=0}^\infty \frac{q^{n^2}}{(q;q)_n}=\frac{1}{(q,q^4;q^5)_\infty}, \quad \sum_{n=0}^\infty \frac{q^{n(n+1)}}{(q;q)_n}=\frac{1}{(q^2,q^3;q^5)_\infty}.
\end{align}
Here and throughout this paper, we assume that $|q|<1$ for convergence and use the standard $q$-series notation: 
\begin{align}
&(a;q)_\infty:=\prod\limits_{k=0}^{\infty}(1-aq^k), \quad (a;q)_n:=\frac{(a;q)_\infty}{(aq^n;q)_\infty},\quad  n\in \mathbb{Z}, \\
&(a_1,\dots,a_m;q)_n:=(a_1;q)_n\cdots (a_m;q)_n, \quad n\in \mathbb{Z}\cup \{\infty\}.
\end{align}
The identities in \eqref{RR} were later rediscovered by Ramanujan and are therefore commonly known as the Rogers--Ramanujan identities.

Since the pioneering work of Rogers \cite{Rogers1894}, numerous Rogers--Ramanujan type identities have been discovered. One particularly notable contribution is Slater's famous list \cite{Slater}, which comprises 130 such identities derived using the machinery of Bailey pairs. Below we shall focus on the background related to this paper, and we refer the reader to Sills' book \cite{Sills-book} and the references therein for more detailed introduction on this topic.

Many Rogers--Ramanujan type identities admit elegant combinatorial interpretations as partition identities. Let $1\leq \lambda_1\leq\lambda_2\leq\cdots\leq\lambda_s$ be integers. We call
$\pi:\lambda_1+\lambda_2+\cdots+\lambda_s$
a partition of $n$ if $\sum_{j=1}^s\lambda_j=n$. In many examples, the sum side of a Rogers--Ramanujan type identity is the generating function for partitions satisfying certain difference conditions between consecutive parts, whereas the product side enumerates partitions whose parts belong to specified subsets of the positive integers, often specified by congruence classes. For example, for $i=1,2$, the $i$-th identity in \eqref{RR} asserts that the number of partitions of $n$ whose smallest part is at least $i$ and whose consecutive parts differ by at least $2$ equals the number of partitions of $n$ into parts congruent to $\pm i\pmod{5}$. By combining partition-theoretic ideas with $q$-series techniques, the Rogers--Ramanujan identities \eqref{RR} have been generalized to the following multisum identities:
 for $\delta=0,1$, $k\geq 2$ and $1\leq i \leq k$,
\begin{equation}\label{AGB}
\sum_{n_1,\dots,n_{k-1}\geq 0}
\frac{q^{N_1^2+\cdots+N_{k-1}^2+N_i+\cdots+N_{k-1}}}{(q;q)_{n_1}\cdots(q;q)_{n_{k-2}}
(q^{\,2-\delta};q^{\,2-\delta})_{n_{k-1}}}=
\frac{(q^i,q^{2k+\delta-i},q^{2k+\delta};q^{2k+\delta})_\infty}{(q;q)_\infty}
\end{equation}
where $N_j=n_j+n_{j+1}+\cdots+n_{k-1}$. When $\delta=1$, these are the Andrews--Gordon identities of odd modulus $2k+1$ \cite{Andrews1974,Gordon1961}, which generalize the Rogers--Ramanujan identities \eqref{RR}. They were first established by Gordon \cite{Gordon1961} as partition identities, and the analytic form here was subsequently proved by Andrews \cite{Andrews1974}. When $\delta=0$, they reduce to Bressoud's identities of even modulus $2k$ \cite{Bressoud1980}.

In a series of seminal works, Lepowsky and Wilson \cite{Lepowsky-Wilson-1978,Lepowsky-Wilson1982,Lepowsky-Wilson-1984} developed a representation-theoretic approach to the Rogers--Ramanujan identities and their generalizations \eqref{AGB}, based on vertex operators and $Z$-algebras associated with the affine Lie algebra $A_1^{(1)}$. They interpreted the sum sides in terms of bases for vacuum spaces constructed from principally twisted $Z$-operators, while the product sides emerged from principally specialized characters of standard $A_1^{(1)}$-modules, thereby providing a conceptual Lie-theoretic interpretation of \eqref{AGB}. Their work furnished a purely vertex-operator-theoretic proof of the Rogers--Ramanujan identities and established a framework for treating the corresponding higher-level identities; a complete $Z$-algebra proof of the latter was subsequently developed by Meurman and Primc \cite{MP}.

These developments inspired the search for analogous identities associated with other affine Lie algebras. For example, Capparelli \cite{Capparelli-thesis} discovered two conjectural partition identities in the course of studying level-three standard modules for $A_2^{(2)}$ by means of $Z$-algebras. These identities were first proved by Andrews \cite{Andrews1994} and later by Capparelli \cite{Capparelli1996}. As discussed in \cite{KR-2019}, Rogers--Ramanujan type identities associated with the affine Lie algebras $A_3^{(2)}$, $A_5^{(2)}$, $A_7^{(2)}$, and $A_{11}^{(2)}$, together with their $Z$-algebraic interpretations, have also been investigated in several works including \cite{Bos-Misra,Kanade2018,Nandi-thesis}.

Inspired by the above works, Kanade and Russell \cite{KR-2015} conducted an extensive search for new partition identities and conjectured six such identities labeled $I_k$ for $1\leq k\leq 6$. The first four identities, $I_1,I_2,I_3$ and $I_4$, have modulus $9$, whereas $I_5$ and $I_6$ have modulus $12$. The mod 9 identities are so far still open. We refer the reader to the work of Kur\c{s}ungoz \cite{Kursungoz},  Kanade \cite{Kanade-arXiv}, Chern and Li \cite{Chern-Li}, Uncu and Zudilin \cite{Uncu-Zudilin}, Tsuchioka \cite{Tsuchioka} and  Konenkov \cite{Konenkov} for more details on the progress of the mod 9 conjectures. We will only consider the mod 12 identities of Kanade and Russell.  
For $r=0,1,2$, following \cite{KR-2019} we say that a partition $\pi: \lambda_1+\cdots+\lambda_s$ satisfies \textit{Condition $(r)$} if $\lambda_{t+3}-\lambda_t\geq 3$ and $\lambda_t+\lambda_{t+1}+\lambda_{t+2}\equiv r$ (mod 3) whenever $\lambda_{t+2}-\lambda_t\leq 1$. Identity $I_5$ asserts that the number of partitions of $n$ satisfying Condition (1), in which the part $1$ occurs at most once, equals the number of partitions of $n$ into parts congruent to $1,3,4,6,7,10, 11$ (mod 12). Identity $I_6$ asserts that the number of partitions of $n$ satisfying Condition (2), with smallest part at least $2$ and with the part $2$ occurring at most once, equals the number of partitions of $n$ into parts congruent to $2,3,5,6,7,8,11$ (mod $12$). In his Ph.D. thesis, Russell \cite{Russell-thesis} conjectured two further partition identities of a similar nature labeled $I_{5a}$ and $I_{6a}$.

Subsequently, Kanade and Russell \cite{KR-2019} discovered thirteen more mod 12 identities of a similar form  labeled $1,2,3,4,4a,5,5a,6,6a,7,7a,8,8a$. Their investigation is motivated by searching for Rogers--Ramanujan type identities associated with the level-two characters of the affine Lie algebra $A_9^{(2)}$. Let
\begin{align}\label{H-sum-defn}
H(u,v,w)&:=\sum_{i,j,k\geq 0} \frac{(-1)^kq^{3k(k-1)+(i+2j+3k)(i+2j+3k-1)}u^iv^jw^k}{(q;q)_i(q^4;q^4)_j(q^6;q^6)_k}.
\end{align} 
The analytic forms of their identities 1, 2, 3 are
\begin{align}
H(q,q^6,q^8)&=\frac{1}{(q,q^4,q^6,q^8,q^{11};q^{12})_\infty}, \label{KR1} \\
H(q^2,q^2,q^9)&=\frac{(q^6;q^{12})_\infty}{(q^2,q^3,q^4;q^6)_\infty}, \label{KR2} \\
H(q^4,q^6,q^{15})&=\frac{1}{(q^4,q^5,q^6,q^7,q^8;q^{12})_\infty}. \label{KR3}
\end{align}
The product sides of the identities \eqref{KR1}--\eqref{KR3} are principally specialized vacuum-space characters for the level-2 $A_9^{(2)}$ modules $L(\Lambda_0+\Lambda_1)$, $L(\Lambda_3)$ and $L(\Lambda_5)$ (see \cite{Bos}), respectively.  Kanade and Russell then provided six more similar identities with asymmetric product sides for $H(u,v,w)$ with $(u,v,w)=(q,q^3,q^6)$, $(q^2,q^{-1},q^6)$, $(q,1,q^3)$, $(q^2,q^4,q^9)$, $(q,q,q^6)$, and $(q^3,q^5,q^{12})$ labeled 4, 4a, 5, 5a, 6 and 6a, respectively. Recently, Ito \cite{Ito} gave a Lie theoretic interpretation to the identities \eqref{KR1}--\eqref{KR3} by constructing $Z$-monomial bases indexed exactly by the Kanade–Russell partition sets.

Kanade and Russell \cite{KR-2019} also revisited four old identities in \cite{KR-2015,Russell-thesis}. They gave analytic forms for the aforementioned identities $I_{5}$, $I_6$ in \cite{KR-2015} and $I_{5a}$, $I_{6a}$ in \cite{Russell-thesis}. For example, the identities $I_5$ and $I_6$ can be restated as
\begin{align}
\sum_{i,j,k\geq 0}\frac{q^{(i+2j+3k)(i+2j+3k-1)/2+j^2+i+2j+4k}}{(q;q)_i(q^2;q^2)_j(q^3;q^3)_k}&=\frac{1}{(q;q^3)_\infty (q^3,q^6,q^{11};q^{12})_\infty}, \label{I5} \\
\sum_{i,j,k\geq 0}\frac{q^{(i+2j+3k)(i+2j+3k-1)/2+j^2+2i+4j+5k}}{(q;q)_i(q^2;q^2)_j(q^3;q^3)_k}&=\frac{1}{(q^2;q^3)_\infty (q^3,q^6,q^7;q^{12})_\infty}.  \label{I6}
\end{align}
These can be expressed in equivalent forms in terms of quadruple sums.  Let
\begin{align}\label{F-quadruple-defn}
F(u,v,w,t)=F(u,v,w,t;q):=\sum_{i,j,k,\ell \geq 0}\frac{(-1)^{\ell}u^{i}v^{j}w^{k}t^{\ell}q^{\binom{i+2j+3k+4\ell}{2}+2\ell^2-2\ell}}{(q;q)_{i}(q^2;q^2)_{j}(q^3;q^3)_{k}(q^4;q^4)_{\ell}}.
\end{align}
Then \eqref{I5} and \eqref{I6} are equivalent to
\begin{align}
 F(q,q^3,q^4,q^6)&=\frac{1}{(q,q^3,q^4,q^6,q^7,q^{10},q^{11};q^{12})_\infty}, \label{I5-quadruple} \\
F(q^2,q^5,q^5,q^{10})&=\frac{1}{(q^2,q^3,q^5,q^6,q^7,q^8,q^{11};q^{12})_\infty}. \label{I6-quadruple}
\end{align} 
The identity \eqref{I6-quadruple} was stated in \cite[(54)]{KR-2019}, whereas \eqref{I5-quadruple} was first stated explicitly by the author in \cite{Wang2023}. Nevertheless, the equivalence of \eqref{I5} and \eqref{I5-quadruple} follows from a straightforward argument analogous to the one leading to \cite[(54)]{KR-2019}. Motivated by the quadruple-sum formulations \eqref{I5-quadruple} and \eqref{I6-quadruple}, Kanade and Russell \cite{KR-2019} varied the linear terms in the exponents of $q$ and conjectured four companion identities, labeled $7$, $7a$, $8$, and $8a$. We remark that Kur\c{s}ung\"oz \cite{Kursungoz} provided alternative sum sides for the identities $I_1,I_2,I_3$ and $I_4$ of modulus $9$, as well as for the identities $I_5$ and $I_6$ of modulus $12$.

Of the seventeen conjectural mod 12 identities discussed in \cite{KR-2019}, eleven have been proved. In particular, Bringmann, Jennings-Shaffer, and Mahlburg \cite{BJSM} established five of the conjectural identities involving $F(u,v,w)$, namely identities $1,2,3,5$, and $5a$, together with the identities $I_5$ and $I_6$, given in \eqref{I5} and \eqref{I6}, respectively. Subsequently, Rosengren \cite{Rosengren} proved all nine conjectural identities involving $F(u,v,w)$ by means of an integral method.

To the best of the author's knowledge, the remaining conjectural identities $I_{5a}$, $I_{6a}$ and $7$, $7$a, $8$ and $8$a have been open until now. The main purpose of this paper is to prove all six conjectures, which we collect in the following two theorems.
\begin{theorem}[Kanade--Russell conjectural identities $I_{5a}$ and $I_{6a}$]\label{thm-I56}
We have
\begin{align}
&\sum_{i,j,k\geq 0} \frac{(1+q^{2i+4j+6k+2}-q^{3i+6j+9k+5})q^{\binom{i+2j+3k}{2}+i+j^2+2j+2k}}{(q;q)_i(q^2;q^2)_j(q^3;q^3)_k}=\frac{1}{(q^2;q^3)_\infty(q,q^6,q^9;q^{12})_\infty}, \tag{$I_{5a}$} \label{I5a}\\
& \sum_{i,j,k\geq 0} \frac{(1+q+q^{2j+3k+2})q^{\binom{i+2j+3k}{2}+3i+j^2+4j+7k}}{(q;q)_i(q^2;q^2)_j(q^3;q^3)_k}=\frac{1}{(q;q^3)_\infty(q^5,q^6,q^9;q^{12})_\infty}.  \tag{$I_{6a}$}  \label{I6a}
\end{align}
\end{theorem}
The combinatorial forms of these identities were originally stated in Russell's thesis \cite{Russell-thesis}, and the above analytical forms were given by Kanade and Russell in \cite[(57) and (61)]{KR-2019}.

\begin{theorem}[Kanade--Russell conjectural identities 7, 7a, 8 and 8a]\label{thm-I78}
Let $F(u,v,w,t)$ be defined in \eqref{F-quadruple-defn}. We have
\begin{align}
F(q,q^3,q^6,q^8)&=\frac{1}{(q,q^3,q^4,q^6,q^8,q^9,q^{11};q^{12})_\infty}, \tag{KR7} \label{KR7} \\
F(q^3,q^5,q^6,q^{12};q)&=\frac{1}{(q^3,q^4,q^5,q^6,q^7,q^8,q^9;q^{12})_\infty}, \tag{KR7a} \label{KR7a} \\
F(q^2,q^3,q^5,q^8)&=\frac{1}{(q^2,q^3,q^4,q^5,q^8,q^9,q^{11};q^{12})_\infty}  \tag{KR8}\label{KR8}
\end{align}
and
\begin{align}
&F(q^2,q^3,q^4,q^8)+qF(q^3,q^5,q^7,q^{12})+q^3F(q^4,q^7,q^{10},q^{16})\nonumber \\
& \quad -q^7F(q^5,q^9,q^{13},q^{20})  =\frac{1}{(q,q^3,q^4,q^7,q^8,q^9,q^{10};q^{12})_\infty}. \tag{KR8a} \label{KR8a}
\end{align}
\end{theorem}
We note that the product sides of \eqref{KR7} and \eqref{KR7a} are symmetric and hence become modular after multiplication by suitable powers of $q$. By contrast, the infinite products appearing in \eqref{I5a}, \eqref{I6a}, \eqref{KR8}, and \eqref{KR8a} are asymmetric.

To state the partition-theoretic interpretations of these identities given by Kanade and Russell \cite{KR-2019}, we introduce several forbidden partition patterns. By a subpartition we mean a contiguous block of parts in the original partition. For $r\geq 1$, let $P_1,\dots,P_{11}$ denote the following families of subpartition patterns, where $r$ ranges over all positive integers:
\begin{equation}\label{pattern}
\begin{split}
&P_1: r+r+r, \quad P_2: r+r+(r+1), \quad P_3: r+(r+1)+(r+1), \\
&P_4: r+r+r+r, \quad P_5: r+r+r+(r+2), \\
&P_6: r+r+(r+1)+(r+2), \quad P_7: r+(r+1)+(r+2)+(r+2), \\
&P_8: r+(r+1)+(r+2)+(r+3), \quad  P_9: r+(r+2)+(r+2)+(r+2),  \\
&P_{10}: r+r+(r+2)+(r+2)+(r+3), \\
&P_{11}: r+(r+1)+(r+1)+(r+3)+(r+3).
\end{split}
\end{equation}

Combining Theorems \ref{thm-I56} and \ref{thm-I78} with the partition-theoretic interpretations of the corresponding analytic sum-sides established in \cite{KR-2019}, we obtain the following partition identities.
\begin{corollary}
For every integer $n\geq 0$ and for each row of Table \ref{tab-5678}, the number of partitions of $n$ satisfying
the restrictions in the second column is equal to the number of partitions of $n$ all of whose parts are congruent modulo $12$ to one of the residue classes in the third column.
\end{corollary}
\begin{table}[htbp]
    \centering
    \caption{Partition-theoretic interpretations of Theorems
\ref{thm-I56} and \ref{thm-I78}}\label{tab-5678}
    \renewcommand{\arraystretch}{1.2}
    \begin{tabular}{c>{\raggedright\arraybackslash}p{8.2cm}c}
    \toprule
    Identity & Difference and pattern restrictions & Residue classes mod 12 \\
    \midrule
    $I_{5a}$    & Satisfies Condition (2) and avoids $1+2+2$ &  $1,2,5,6,8,9,11$ \\
    $I_{6a}$  & Satisfies Condition (1) and avoids $2$ & $1,4,5,6,7,9,10$ \\
    7
    & Avoids $1+1$, $2+2+2$, $P_2$, $P_3$, $P_4$, $P_5$, $P_9$
    & $1,3,4,6,8,9,11$ \\

    7a
    & Each part $\geq 3$ and avoids $P_2$, $P_3$, $P_4$, $P_5$, $P_9$
    & $3,4,5,6,7,8,9$ \\

    8
    & Avoids $1$, $P_1$, $P_2$, $P_7$, $P_8$, $P_{11}$
    & $2,3,4,5,8,9,11$ \\

    8a
    & Avoids $1+1$, $1+2+3$, $2+2+3$, $P_1$, $P_3$,
      $P_6$, $P_8$, $P_{10}$
    & $1,3,4,7,8,9,10$ \\
    \bottomrule
    \end{tabular}
\end{table}

The proof of Theorem \ref{thm-I56} is inspired by the approach of Bringmann et al.~\cite{BJSM} used to prove the Kanade--Russell identities  \eqref{I5}--\eqref{I6}. It is worth noting that, in \cite[Section 5]{BJSM}, they obtained a partial reduction of \eqref{I5a} by decomposing its left-hand side into three sums and reducing two of them to single series. Our proof does not rely on this decomposition. Instead, we treat each summand as a single entity and derive suitable linear recurrences. These recurrences yield $q$-difference equations for $x$-parameterized extensions of the original sums. Solving the resulting equations reduces each multisum to a single ${}_2\phi_1$ series, which can then be evaluated using the Bailey--Daum summation formula (see \eqref{Bailey-Daum}), yielding the desired infinite-product expressions.

The proof of Theorem \ref{thm-I78} is based on the integral method and motivated by Rosengren's proofs of nine Kanade--Russell identities for $H(u,v,w)$ defined in \eqref{H-sum-defn}. The proof strategy was briefly outlined in the author's earlier work \cite[Section 4.2]{Wang2023}, and the idea is similar to the proofs of the Andrews--Uncu conjecture given by Chern \cite{Chern} and the author \cite{Wang2023}. It relies on some general theory  in the book of Gasper and Rahman \cite[Chapter 4]{GR-book}. Specifically, we first express the sum side as contour integral of some infinite products, then we evaluate this integral by residue calculus. Each resulting residue factors into an infinite product and a single-sum cubic $q$-hypergeometric series. Here, a series is called \emph{cubic} if, after the smallest base occurring in its summand is normalized to $q$, the largest base is $q^3$. We then evaluate these cubic series in infinite-product form. The main difficulty therefore lies in the evaluation of the cubic series themselves. A few follow directly from known cubic summation formulas, whereas the remaining cases require more intricate analysis. As key ingredients in the proof of Theorem \ref{thm-I78}, we establish several new cubic Rogers--Ramanujan type identities, stated in Theorems \ref{thm-cubic-1} and \ref{thm-cubic-2}.

The rest of this paper is organized as follows. In Section \ref{sec-pre} we collect some useful formulas and an integral lemma which will be fundamental to the proof of Theorem \ref{thm-I78}. We will prove Theorem \ref{thm-I56} in Section \ref{sec-triple}.  Section \ref{sec-cubic} is devoted to establishing some cubic identities arising in the proof of Theorem \ref{thm-I78}. Finally, in Section \ref{sec-quadruple} we provide proofs for the quadruple sum identities \eqref{KR7}--\eqref{KR8a}.

\section{Preliminaries}\label{sec-pre}

First, we need Euler's $q$-exponential identities
\begin{align}
\sum_{n=0}^\infty \frac{z^n}{(q;q)_n}&=\frac{1}{(z;q)_\infty}, \quad |z|<1,  \label{Euler}\\
\sum_{n=0}^\infty \frac{q^{\binom{n}{2}} z^n}{(q;q)_n}&=(-z;q)_\infty. \label{Euler2}
\end{align}
These two identities are corollaries of the $q$-binomial theorem
\begin{align}\label{q-binomial}
    \sum_{n=0}^\infty \frac{(a;q)_n}{(q;q)_n}z^n=\frac{(az;q)_\infty}{(z;q)_\infty}, \quad |z|<1.
\end{align}
We also need the Jacobi triple product identity
\begin{align}\label{Jacobi}
(q,z,q/z;q)_\infty=\sum_{n=-\infty}^\infty (-1)^nq^{\binom{n}{2}}z^n.
\end{align}

For compactness, we denote 
$$J_m:=(q^m;q^m)_\infty, \quad J_{a,m}:=(q^a,q^{m-a},q^m;q^m)_\infty.$$
It is easy to verify that
\begin{align}\label{eq-J-eta}
J_{1,2}=\frac{J_1^2}{J_2}, \quad J_{1,3}=J_1, \quad J_{1,4}=\frac{J_1J_4}{J_2}, \quad J_{1,6}=\frac{J_1J_6^2}{J_2J_3}.
\end{align}

The basic hypergeometric series ${}_r\phi_s$ is defined as
\begin{align}
{}_r\phi_s\!\left[\begin{matrix} a_1,\dots, a_r \\ b_1, \dots,  b_s\end{matrix}; q,z \right]:=\sum_{n=0}^\infty\frac{(a_1,\dots,a_r;q)_n}{(q,b_1,\dots,b_s;q)_n}((-1)^nq^{n(n-1)/2})^{1+s-r}z^n.
\end{align}
A key to the proof of Theorem \ref{thm-I56} is the Bailey--Daum summation formula \cite[(\uppercase\expandafter{\romannumeral 2}.9)]{GR-book}:
\begin{align}\label{Bailey-Daum}
 {}_2\phi_1\!\left[
 \begin{matrix} a,b\\ aq/b\end{matrix};q,-q/b\right]
 =\frac{(aq,aq^2/b^2,q^2;q^2)_\infty}
 {(-q/b,aq/b,q;q)_\infty}.                                 
\end{align}

For a Laurent series $f(z)=\sum_{n\in \mathbb{Z}}^\infty a(n)z^n$, we shall use $[z^n]f(z)$ to denote the coefficient of $z^n$. That is, $[z^n]f(z)=a(n)$. Let $f$ be analytic in an annulus containing a positively oriented simple closed contour $\gamma$, with 0 lying in the bounded component of $\mathbb{C}\backslash \gamma$. We will use the following simple fact usually without mention
\begin{align}\label{int-constant}
\oint_{\gamma} f(z) \frac{dz}{2\pi \ii z}=[z^0]f(z).
\end{align}

\begin{lemma}\label{lem-integral}
Suppose that
\begin{equation}\label{eq-product-I}
I(z)=\frac{(a_1z,\dots,a_Az,b_1/z,\dots,b_B/z;q)_\infty}{(c_1z,\dots,c_Cz,d_1/z,\dots,d_D/z;q)_\infty}
\end{equation}
has only simple poles, and let $p(z)$ be a polynomial.  Let $\gamma$ be a
positively oriented contour such that the poles of 
$(c_1z,\dots,c_Cz;q)_\infty^{-1}$
lie outside $\gamma$, whereas the origin and the poles of
$(d_1/z,\dots,d_D/z;q)_\infty^{-1}$
lie inside $\gamma$. When $C>A$, of if $C=A$ and $\left|\frac{a_1\cdots a_A}{c_1\cdots c_C}\right|<1$, we have
\begin{align}\label{eq-integral-I}
\oint_{\gamma}p(z)I(z)\frac{dz}{2\pi\mathrm{i}z}&=\sum_{s=1}^{C}K_s\sum_{n=0}^{\infty}\Phi_{s,n} p\left(\frac{q^{-n}}{c_s}\right), 
\end{align}
where
\begin{align}\label{eq-Ks}
K_s=\frac{(b_1c_s,\dots,b_Bc_s,a_1/c_s,\dots,a_A/c_s;q)_\infty}{(q,d_1c_s,\dots,d_Dc_s,\{c_j/c_s:1\leq j\leq C,\ j\neq s\};q)_\infty},
\end{align}
and
\begin{align}
\Phi_{s,n}&=\frac{(d_1c_s,\dots,d_Dc_s,qc_s/a_1,\dots,qc_s/a_A;q)_n}{(q,b_1c_s,\dots,b_Bc_s,\{qc_s/c_j:1\leq j\leq C,\ j\neq s\};q)_n}\notag\\
&\quad\times \left(-c_sq^{(n+1)/2}\right)^{n(C-A)}\left(\frac{a_1\cdots a_A}{c_1\cdots c_C}\right)^n.
\label{eq-Phi}
\end{align}
\end{lemma}
This lemma follows immediately from \cite[(4.10.6)]{GR-book}.

For any positive integer $m$, we recall Atkin's $U$-operator defined by
\begin{align}
\sum_{n=-\infty}^\infty a(n)q^n \mid U_m:=\sum_{n=-\infty}^\infty a(mn)q^n.
\end{align}

After replacing $q$ by $-q$, the identity in \cite[Entry 25 (vi)]{Notebook3} can be rewritten as
\begin{align}\label{phi-square}
\frac{J_1^4}{J_2^4}=\frac{J_4^{10}}{J_2^6J_8^4}-4q\frac{J_8^4}{J_2^2J_4^2}.    
\end{align}
From \cite[(3.1)]{Toh} we find
\begin{align}\label{eq-Toh}
\frac{J_2^3}{J_1^3}&=\left(\frac{J_6}{J_3}+12q^3\frac{J_6^2J_{18}^5}{J_3^6J_9}\right)
+3q\frac{J_6^4J_9^5}{J_3^8J_{18}}+6q^2\frac{J_6^3J_9^2J_{18}^2}{J_3^7}.
\end{align}
It follows that
\begin{align}
q^{-1}(q;q^2)_\infty^4 |U_2&=-4\frac{J_4^{4}}{J_1^2J_2^2}, \label{J2J1-U2} \\
q^{-1}(-q;q)_\infty^3|U_3&=3\frac{J_2^4J_3^5}{J_1^8J_6}. \label{J2J1-U3}
\end{align}

For the purpose of evaluating some cubic series, we recall some quadratic, cubic and sextic summation formulas in the literature. We need two quadratic summation formulas from Rahman's work \cite[(1.8) and (4.6)]{Rahman1993}:
\begin{align}
&\sum_{n=0}^{\infty}\frac{(a;q^2)_n(1-aq^{3n})(b,c,aq/(bc);q)_n}{(q;q)_n(1-a)(aq^2/b,aq^2/c,bcq;q^2)_n}q^{n(n+1)/2}
 =\frac{(aq^2,bq,cq,aq^2/(bc);q^2)_\infty}{(q,aq^2/b,aq^2/c,bcq;q^2)_\infty},  \label{eq-Rahman18} \\
 &\sum_{n=0}^{\infty}\frac{(a;q)_n(1-aq^{3n})(d,q/d;q)_n(b;q^2)_n}{(q^2;q^2)_n(1-a)(aq^2/d,adq;q^2)_n(aq/b;q)_n}
 \left(\frac{aq}{b}\right)^n q^{\binom{n}{2}}  \nonumber \\
 &=\frac{(aq,aq^2,adq/b,aq^2/(bd);q^2)_\infty}{(aq/b,aq^2/b,aq^2/d,adq;q^2)_\infty}. \label{eq-Rahman46}
\end{align}

Recall the following cubic summation formula (see \cite[(1.7)]{Chern} and \cite[(1.29)]{Wang2023}):
\begin{align}\label{id-cubic-C}
\sum_{n=0}^\infty \frac{(a;q)_n(a^{-1}q^2;q^2)_n}{(a^2q;q^2)_n(q^3;q^3)_n}(-1)^na^nq^{n(n+1)/2}=\frac{(aq;q^2)_\infty (a^3q^3;q^6)_\infty}{(a^2q;q^2)_\infty (q^3;q^6)_\infty}.
\end{align}
As pointed out by the author \cite[(3.21)]{Wang2023}, this formula is a special case of Rahman's sextic summation formula \cite{Rahman1989} (see also \cite[Exercise 3.29(i)]{GR-book}.

We will also use the following formula of Chu and Wang \cite[Corollary 9]{Chu-Wang}
\begin{align}\label{eq-Chu}
\sum_{k=0}^\infty \frac{(q^{\frac{1}{2}};q)_k(q^\frac{3}{2}a;q^2)_k}{(q^2;q^2)_k(q^3a;q^3)_k}(-1)^kq^{(k^2+2k)/2}
=\frac{(q^{\frac{3}{2}}a;q^2)_\infty (q^{\frac{3}{2}};q^3)_\infty (q^6;q^6)_\infty}{(q^2;q^2)_\infty (q^3a;q^3)_\infty (q^{\frac{3}{2}}a;q^6)_\infty}.
\end{align}


\section{Proofs of the triple sum identities}\label{sec-triple}
In this section, we shall present a proof for Theorem \ref{thm-I56}. The approach is motivated by the work of Bringmann et al.~\cite{BJSM} especially their proofs of \eqref{I5}--\eqref{I6} given in \cite[Sections 4.9--4.10]{BJSM}. Note that the situation here is slightly different as we need to deal with three triple sums together while the identities proved in \cite{BJSM}  involve only one triple sum. 

We define
\begin{align*}
 G(u,v,w):=\sum_{i,j,k\geq 0}\frac{q^{\binom{i+2j+3k}{2}+j^2}u^iv^jw^k}{(q;q)_i(q^2;q^2)_j(q^3;q^3)_k}.
\end{align*}
We denote the left side of  \eqref{I5a} and \eqref{I6a} by $S_1(q)$ and $S_2(q)$, respectively. Then we have
\begin{align}
S_1(q)&=G(q,q^2,q^2)+q^2\big(G(q^3,q^6,q^8)-q^3G(q^4,q^8,q^{11})\big), \label{S1-G-start} \\
S_2(q)&=(1+q)G(q^3,q^4,q^7)
   +q^2G(q^3,q^6,q^{10}). \label{S2-G}
\end{align}
We have
\begin{align}\label{G-diff}
   &G(u,v,w)-G(uq,v,w)= \sum_{i,j,k\geq 0}\frac{q^{\binom{i+2j+3k}{2}+j^2}u^i(1-q^i)v^jw^k}{(q;q)_i(q^2;q^2)_j(q^3;q^3)_k} \nonumber \\
&=\sum_{i,j,k\geq 0}\frac{q^{\binom{i+2j+3k+1}{2}+j^2}u^{i+1}v^jw^k}{(q;q)_i(q^2;q^2)_j(q^3;q^3)_k}=uG(qu,q^2v,q^3w).
\end{align}
Here for the second equality we replaced $i$ by $i+1$ and used the following simple fact for simplification:
\begin{align}\label{binom-split}
    \binom{m+n}{2}=\binom{m}{2}+mn+\binom{n}{2}.
\end{align}
From \eqref{G-diff} with $(u,v,w)=(q^3,q^6,q^8)$ we can rewrite \eqref{S1-G-start} as
\begin{align}\label{S1-G}
S_1(q)=G(q,q^2,q^2)+q^2G(q^4,q^6,q^8).
\end{align}

Using \eqref{binom-split} and summing over $i$ first using \eqref{Euler2}, we have
\begin{align}
    G(u,v,w)=\sum_{j,k\geq 0}\frac{(-q^{2j+3k}u;q)_\infty q^{\binom{2j+3k}{2}+j^2}v^jw^k}{(q^2;q^2)_j(q^3;q^3)_k}.
\end{align}
Hence for nonnegative integers $a,b,c$, we have
\begin{align}\label{G-simplify}
    G(q^a,q^b,q^c)=(-q;q)_\infty \sum_{j,k\geq 0} \frac{q^{\binom{2j+3k}{2}+j^2+bj+ck}}{(-q;q)_{2j+3k+a-1}(q^2;q^2)_j(q^3;q^3)_k}.
\end{align}

We shall evaluate $S_1(q)$ and $S_2(q)$ through \eqref{S1-G} and \eqref{S2-G} by recurrences.  

\subsection{Proof of the $I_{5a}$ identity}
From \eqref{S1-G} and \eqref{G-simplify} we have
\begin{align}\label{S1-start}
    S_1(q)&=(-q;q)_\infty \Big(\sum_{j,k\geq 0} \frac{q^{\binom{2j+3k}{2}+j^2+2j+2k}}{(-q;q)_{2j+3k}(q^2;q^2)_j(q^3;q^3)_k} \nonumber \\
    &\quad +q^2\sum_{j,k\geq 0} \frac{q^{\binom{2j+3k}{2}+j^2+6j+8k}}{(-q;q)_{2j+3k+3}(q^2;q^2)_j(q^3;q^3)_k}  \Big) \nonumber \\
    &=(-q;q)_\infty \sum_{j,k\geq 0} \frac{(1+q^{-2j-3k}(1-q^{3k}))q^{\binom{2j+3k}{2}+j^2+2j+2k}}{(q;q)_{2j+3k}(q^2;q^2)_j(q^3;q^3)_k}.
\end{align}
Here for the last equality we replaced $k$ by $k-1$ in the second sum and combined the two sums together.

We define
\begin{equation}\label{eq-def-A}
A(x):=\sum_{j,k\geq 0}\frac{\big(1+q^{-2j-3k}(1-q^{3k})\big)q^{\binom{2j+3k}{2}+j^2+2j+2k}x^{2j+2k}}{(-q;q)_{2j+3k}(q^2;q^2)_j(q^3;q^3)_k}.
\end{equation}
Notice in particular that $A(0)=1$. Then \eqref{S1-start} can be written as
\begin{equation}\label{eq-S-A}
S_1(q)=(-q;q)_\infty A(1).
\end{equation}

We write $M=j+k$. Then $j=M-k$, and a direct calculation gives
\begin{equation}\label{eq-quadratic-A}
 \binom{2j+3k}{2}+j^2+2j+2k
 =3M^2+M+\frac{k(3k-1)}{2}.
\end{equation}
Consequently, if we denote
\begin{equation}\label{eq-def-h}
 h_{k,M}:=\frac{\big(1+q^{-2M-k}(1-q^{3k})\big)q^{k(3k-1)/2}}{(-q;q)_{2M+k}(q^2;q^2)_{M-k}(q^3;q^3)_k}, \quad  a_M=\sum_{k=0}^M h_{k,M},
\end{equation}
then \eqref{eq-def-A} can be written as
\begin{equation}\label{eq-A-coefficients}
 A(x)=\sum_{M\geq 0} a_Mq^{3M^2+M}x^{2M}.
\end{equation}

\begin{lemma}\label{lem-recurrence}
With $a_{-1}:=0$, the coefficients in \eqref{eq-A-coefficients} satisfy, for $M\geq 1$, 
\begin{equation}\label{eq-a-recurrence}
 \begin{split}
 (1-q^{6M})(1-q^{6M-4})a_M=(q^{-2}+q^2-q^{6M-7}-q^{6M-3})a_{M-1}-a_{M-2}.
 \end{split}
\end{equation}
\end{lemma}
Both this lemma and Lemma \ref{lem-b-recurrence} below can be proved using the $q$-Zeilberger algorithm implemented in the MAPLE package {\texttt{QDifferenceEquations}}. For the sake of completeness, we give proofs including explicit telescoping certificates for these lemmas. 
\begin{proof}
We introduce independent variables $s,t$ and define
\begin{align}
 C(t,s)&:=1+\frac{1-t^3}{s^2t}=\frac{s^2t+1-t^3}{s^2t}, \\
 D(t,s)&:=C(t,s)s^2t(1+qs^2t)(1+q^2s^2t)(1+q^3s^2t)(t^2-q^2s^2)(q^4s^2-t^2).
\end{align}
We define the polynomial
\begin{align*}
Q(t,s)&:=q^6s^4+q^7s^6(1+q+q^2)t+s^2(1+q)(q^9s^6-q+1)t^2 \\
 & \quad -q^4s^4(q^{10}s^6-q^8s^6+2q+1)t^3-q^4s^6(1+q+q^2+q^3+q^4)t^4 \\
 &\quad -q^{-1}s^2(q^9s^6-q^2+1)t^5+q^3s^4(1+q)t^6+q^6s^6t^7
\end{align*}
and let
\begin{equation}\label{eq-Rg-defn}
 R(t,s):=\frac{(1-t^3)Q(t,s)}{D(t,s)},  \quad g_{k,M}:=R(q^k,q^M)h_{k,M}.
\end{equation}

With $s=q^M$ and $t=q^k$, from \eqref{eq-def-h} we have
\begin{equation}\label{eq-ratio-k-A}
\frac{h_{k+1,M}}{h_{k,M}}=qt^3\frac{C(qt,s)}{C(t,s)}\frac{1-s^2/t^2}{(1+qs^2t)(1-q^3t^3)}
\end{equation}
and
\begin{equation}\label{eq-ratio-M-A}
 \frac{h_{k,M+1}}{h_{k,M}}=\frac{C(t,qs)}{C(t,s)}\frac{1}{(1+qs^2t)(1+q^2s^2t)(1-q^2s^2/t^2)}.
\end{equation}

From \eqref{eq-Rg-defn}--\eqref{eq-ratio-M-A} we deduce that
\begin{align}
g_{k+1,M}-g_{k,M} &=(1-q^{6M+12})(1-q^{6M+8})h_{k,M+2} \nonumber \\
 & \quad -(q^{-2}+q^2-q^{6M+5}-q^{6M+9})h_{k,M+1} +h_{k,M}.   \label{eq-telescoping}                  
\end{align}
In fact, after dividing by $h_{k,M}$, both sides can be expressed as rational functions of $s,t$, and their difference is easily checked to be identically zero by Maple.

Note that $g_{0,M}=0$. Due to the convention $1/(q^2;q^2)_r=0$ for negative integers $r$,  we have $h_{k,M}=0$ for
$k>M$.  Hence $g_{M+3,M}=0$.  Summing
\eqref{eq-telescoping} for $0\leq k\leq M+2$ therefore gives
\begin{align*}
 0=(1-q^{6M+12})(1-q^{6M+8})a_{M+2}-(q^{-2}+q^2-q^{6M+5}-q^{6M+9})a_{M+1}+a_M.
\end{align*}
This proves \eqref{eq-a-recurrence} for $M\geq2$.  Finally,
\eqref{eq-def-h} gives
\begin{align*}
 a_0=1, \quad a_1=\frac{q^{-2}(1+q^4)}{(1+q)(1-q^6)}.
\end{align*}
This proves the remaining case $M=1$ by direct substitution.
\end{proof}

\begin{proof}[Proof of \eqref{I5a}]
We claim that the function $A$ defined by \eqref{eq-def-A} satisfies
\begin{equation}\label{eq-A-rec}
 A(x)=(1+q^{-4}+x^2q^2+x^2q^6)A(xq^3)-q^{-4}(1+x^2q^7)(1+x^2q^{11})A(xq^6).
\end{equation}
In fact, inserting \eqref{eq-A-coefficients} into \eqref{eq-A-rec} and comparing the coefficient of $x^{2M}$,  after division by $q^{3M^2+M}$, the result is precisely \eqref{eq-a-recurrence}.  

We now solve \eqref{eq-A-rec} directly.  Let
\begin{equation}\label{eq-def-reduce-A}
 A(x)=(-x^2q^5;q^6)_\infty \mathcal{A}(x),
 \quad \mathcal{A}(x)=\sum_{n\geq 0}\alpha_nx^n.
\end{equation}
The relation \eqref{eq-A-rec} becomes
\begin{equation}\label{eq-reduce-A-rec}
 (1+x^2q^5)\mathcal{A}(x)
 =(1+q^{-4}+x^2q^2+x^2q^6)\mathcal{A}(xq^3)-q^{-4}(1+x^2q^7)\mathcal{A}(xq^6).
\end{equation}
Coefficient comparison in \eqref{eq-reduce-A-rec} gives
\begin{equation}\label{eq-alpha-recurrence}
 (1-q^{3n})(1-q^{3n-4})\alpha_n=-q^5(1-q^{3n-9})(1-q^{3n-5})\alpha_{n-2}.
\end{equation}
Since $\alpha_0=A(0)=1$, we have
\begin{align*}
 \alpha_{2m}=\frac{(q^{-3},q;q^6)_m}{(q^2,q^6;q^6)_m}(-q^5)^m.
\end{align*}
Since $A(x)$ is even, we have $\alpha_{2m+1}=0$. Hence
\begin{equation}\label{eq-A-2phi1}
 A(x)=(-x^2q^5;q^6)_\infty{}_2\phi_1\!\left[\begin{matrix}q^{-3},\ q\\ q^2\end{matrix};q^6,-x^2q^5
 \right].
\end{equation}

Replacing $q$ by $q^6$ in \eqref{Bailey-Daum} and setting $a=q^{-3}$ and $b=q$, we obtain
\begin{align}\label{2phi1-special-1}
    {}_2\phi_1\!\left[\begin{matrix}q^{-3},\ q\\ q^2\end{matrix};q^6,-q^5
 \right]=\frac{(q^3,q^7,q^{12};q^{12})_\infty}{(q^2,-q^5,q^6;q^6)_\infty}.
\end{align}
Combining \eqref{2phi1-special-1}, \eqref{eq-A-2phi1} and \eqref{eq-S-A}, we have
\begin{align*}
 S_1(q)=(-q;q)_\infty\frac{(q^3,q^7,q^{12};q^{12})_\infty}{(q^2,q^6;q^6)_\infty}
 =\frac{1}{(q^2;q^3)_\infty(q,q^6,q^9;q^{12})_\infty}.
\end{align*}
This proves the identity \eqref{I5a}.
\end{proof}

\subsection{Proof of the $I_{6a}$ identity}

From \eqref{S2-G} and \eqref{G-simplify} we have
\begin{equation}\label{eq-S2-B}
 S_2(q)=(-q^3;q)_\infty B(1),
\end{equation}
where
\begin{equation}\label{eq-def-B}
 B(x):=\sum_{j,k\geq 0}
 \frac{(1+q+q^{2j+3k+2})
 q^{\binom{2j+3k}{2}+j^2+4j+7k}x^{2j+2k}}
 {(-q^3;q)_{2j+3k}(q^2;q^2)_j(q^3;q^3)_k}.
\end{equation}
We write $M=j+k$. Then $j=M-k$, and a direct calculation from \eqref{eq-quadratic-A} gives
\begin{equation}\label{eq-quadratic-B}
 \binom{2j+3k}{2}+j^2+4j+7k =3M^2+3M+\frac{k(3k+5)}2.
\end{equation}
Consequently, if we denote
\begin{equation}\label{eq-def-bM}
 b_M:=\sum_{k=0}^{M}\widehat{h}_{k,M},\quad
\widehat{h}_{k,M}:=\frac{(1+q+q^{2M+k+2})q^{k(3k+5)/2}}{(-q^3;q)_{2M+k}(q^2;q^2)_{M-k}(q^3;q^3)_k},
\end{equation}
then
\begin{equation}\label{eq-B-coefficients}
B(x)=\sum_{M\geq 0}b_Mq^{3M^2+3M}x^{2M}.
\end{equation}

\begin{lemma}\label{lem-b-recurrence}
With $b_{-1}:=0$, the coefficients in \eqref{eq-def-bM} satisfy, for $M\geq 1$,
\begin{align}\label{eq-b-recurrence}
 (1-q^{6M})(1-q^{6M+4})b_M=(1+q^4-q^{6M-1}-q^{6M+3})b_{M-1}-q^4b_{M-2}.
\end{align}
\end{lemma} 
\begin{proof}
We introduce independent variables $s,t$ and define
\begin{align*}
 \widehat{C}(t,s)&:=1+q+q^2s^2t, \\
 \widehat{D}(t,s)&:=\widehat{C}(t,s)(1+q^3s^2t)(1+q^4s^2t)(1+q^5s^2t)(t^2-q^2s^2)(q^4s^2-t^2).
\end{align*}
We define the polynomial
\begin{align*}
&  \widehat{Q}(t,s):=q^{10}s^4(1+q)+q^{12}s^6(1+q+q^2+q^3+q^4)t\\
 &\quad +\left[-q^4(q-1)(q+1)^2s^2+q^{17}(1+q)s^8\right]t^2
 -\left[q^{10}(1+q)^2s^4+q^{22}(q^2-1)s^{10}\right]t^3\\
 &\quad -q^{13}s^6(1+q)(1+q^2)t^4-q^{18}s^8t^5
\end{align*}
and let
\begin{equation}\label{eq-wR-wg-defn}
 \widehat{R}(t,s):=\frac{(1-t^3) \widehat{Q}(t,s)}{\widehat{D}(t,s)},
 \quad
 \widehat{g}_{k,M}:=\widehat{R}(q^k,q^M)\widehat{h}_{k,M}.
\end{equation}
With $s=q^M$ and $t=q^k$, from \eqref{eq-def-bM} we have
\begin{equation}\label{eq-ratio-k-B}
 \frac{\widehat{h}_{k+1,M}}{\widehat{h}_{k,M}}
 =\frac{q^4t(t^2-s^2)\widehat{C}(qt,s)}
 {\widehat{C}(t,s)(1+q^3s^2t)(1-q^3t^3)},
\end{equation}
and
\begin{equation}\label{eq-ratio-M-B}
 \frac{\widehat{h}_{k,M+1}}{\widehat{h}_{k,M}}
 =\frac{\widehat{C}(t,qs)}
 {\widehat{C}(t,s)(1+q^3s^2t)(1+q^4s^2t)(1-q^2s^2/t^2)}.
\end{equation}
From \eqref{eq-wR-wg-defn}--\eqref{eq-ratio-M-B}, we obtain the identity
\begin{align}
 \widehat{g}_{k+1,M}-\widehat{g}_{k,M}
 &=(1-q^{6M+12})(1-q^{6M+16})\widehat{h}_{k,M+2}\notag\\
 &-(1+q^4-q^{6M+11}-q^{6M+15})\widehat{h}_{k,M+1}
   +q^4\widehat{h}_{k,M}.  \label{eq-wh-telescoping}
\end{align}
In fact, after dividing by $\widehat{h}_{k,M}$, both sides can be expressed as rational functions of $s,t$, and their difference is easily checked to be identically zero by Maple.

Note that $\widehat{g}_{0,M}=0$. It is easy to see that $\widehat{h}_{k,M}=0$ for $k>M$. Hence $\widehat{g}_{M+3,M}=0$.  Summing \eqref{eq-wh-telescoping} over $0\leq k\leq M+2$ gives
\begin{align*}
 0=(1-q^{6M+12})(1-q^{6M+16})b_{M+2}-(1+q^4-q^{6M+11}-q^{6M+15})b_{M+1}+q^4b_M.
\end{align*}
Replacing $M+2$ by $M$, we obtain \eqref{eq-b-recurrence}.
\end{proof}

\begin{proof}[Proof of \eqref{I6a}]
We claim that the function $B$ satisfies 
\begin{align}\label{eq-B-rec}
 B(x)=(1+q^4+x^2q^6+x^2q^{10})B(xq^3)-q^4(1+x^2q^7)(1+x^2q^{11})B(xq^6).
\end{align}
In fact, the initial value follows from \eqref{eq-def-bM}.  Upon inserting \eqref{eq-B-coefficients} into \eqref{eq-B-rec}, the coefficient of $x^{2M}$ is precisely \eqref{eq-b-recurrence}.  
Let
\begin{align}\label{eq-def-reduce-B}
 B(x)=(-x^2q^5;q^6)_\infty \mathcal{B}(x),
 \quad \mathcal{B}(x)=\sum_{n\geq 0}\beta_nx^n.
\end{align}
Then \eqref{eq-B-rec} becomes
\begin{align}\label{eq-reduce-B-rec}
 (1+x^2q^5)\mathcal{B}(x)=(1+q^4+x^2q^6+x^2q^{10})\mathcal{B}(xq^3)-q^4(1+x^2q^7)\mathcal{B}(xq^6).
\end{align}
Equating coefficients in \eqref{eq-reduce-B-rec} gives
\begin{align}\label{eq-beta-recurrence}
 (1-q^{3n})(1-q^{3n+4})\beta_n=-q^5(1-q^{3n-5})(1-q^{3n-1})\beta_{n-2}.
\end{align}
Note that $\beta_0=B(0)=1+q+q^2$, and \eqref{eq-beta-recurrence}
therefore yields
\begin{equation}\label{eq-beta-even}
 \beta_{2m}=(1+q+q^2)\frac{(q,q^5;q^6)_m}{(q^6,q^{10};q^6)_m}(-q^5)^m.
\end{equation}
Since $B$ is even, we have $\beta_{2m+1}=0$. It follows from \eqref{eq-def-reduce-B} and \eqref{eq-beta-even} that
\begin{equation}\label{eq-B-2phi1}
 B(1)=(1+q+q^2)(-q^5;q^6)_\infty{}_2\phi_1\!\left[\begin{matrix}q,\ q^5\\ q^{10}\end{matrix};q^6,-q^5 \right].
\end{equation}

Replacing $q$ by $q^6$ and then setting $a=q^5$ and $b=q$ in
\eqref{Bailey-Daum}, we obtain
\begin{equation}\label{eq-2phi1-evaluation}
 {}_2\phi_1\!\left[\begin{matrix}q,\ q^5\\ q^{10}\end{matrix};q^6,-q^5
 \right]
 =\frac{(q^{11},q^{15},q^{12};q^{12})_\infty}
 {(q^{10},-q^5,q^6;q^6)_\infty}.
\end{equation}
Combining \eqref{eq-S2-B}, \eqref{eq-B-2phi1} and
\eqref{eq-2phi1-evaluation}, we have
\begin{align*}
 S_2(q)=(1+q+q^2)(-q^3;q)_\infty(-q^6;q^6)_\infty
 \frac{(q^{11},q^{15};q^{12})_\infty}{(q^{10};q^6)_\infty}
 =\frac{1}{(q;q^3)_\infty(q^5,q^6,q^9;q^{12})_\infty}.
\end{align*}
This proves the identity \eqref{I6a}.  
\end{proof}
\begin{rem}
The steps \eqref{eq-def-reduce-A}--\eqref{eq-A-2phi1} and \eqref{eq-def-reduce-B}--\eqref{eq-B-2phi1} also follow directly from \cite[Proposition 2.4]{BJSM}. We include the above steps for the sake of completeness.
\end{rem}

\section{Some cubic Rogers--Ramanujan type identities}\label{sec-cubic}
In this section, we establish two families of cubic Rogers--Ramanujan type identities which will play a key role in the proof of Theorem \ref{thm-I78}.
\begin{theorem}\label{thm-cubic-1}
We have
\begin{align} 
&\sum_{n=0}^\infty \frac{(q;q^2)_n(-q^2;q^4)_n}{(q^4;q^4)_n(-q^3;q^6)_n(1+q^{2n+1})}(-1)^nq^{n^2+2n}=\frac{J_{12}J_{1,12}J_{3,12}J_{10,24}}{J_{2,12}J_6^2J_8}, \label{cubic-id-1} \\
&\sum_{n=0}^\infty \frac{(q;q^2)_n(-q^2;q^4)_n}{(q^4;q^4)_n(-q^3;q^6)_{n+1}}(-1)^nq^{n^2+4n}=\frac{J_{12}J_{3,12}J_{5,12}J_{2,24}}{J_{2,12}J_6^2J_8}, \label{cubic-id-2}  \\
&\sum_{n=0}^\infty \frac{(-q;q)_n(-q^2;q^2)_n}{(q;q^2)_{n+1}(q^3;q^3)_{n+1}}q^{(n^2+5n)/2}=\frac{J_8J_{12}^5J_{2,24}}{J_4J_6J_{1,12}J_{2,12}J_{3,12}^2J_{5,12}}, \label{cubic-id-3} \\
&\sum_{n=1}^\infty \frac{(-q;q)_{n-1}(-q^2;q^2)_{n-1}}{(q^3;q^3)_{n-1}(q;q^2)_n (1-q^n)}q^{n(n+1)/2}=\frac14\left(
 \frac{J_2^3J_6^3}{J_1^2J_3^2J_4J_{12}}
 +\frac{2J_4^3J_6}{J_1J_2J_3J_{12}}-3\right).  \label{cubic-id-4}
\end{align}
\end{theorem}
\begin{proof}
(1) We denote the left side of \eqref{cubic-id-1} and \eqref{cubic-id-2} by $S_1(q)$ and $S_2(q)$, respectively.   We have
\begin{align}\label{BS-sum}
    S_1(q)+qS_2(q)=\sum_{n=0}^\infty \frac{(q;q^2)_n(-q^2;q^4)_{n+1}}{(q^4;q^4)_n(-q^3;q^6)_{n+1}}(-1)^nq^{n^2+2n}=\frac{J_3^2J_4J_{24}}{J_2J_6^2J_8}.
\end{align}
Here the last equality follows from \eqref{eq-Chu}  with $q$ replaced by $q^2$ and $a=-q^3$. 

Next, note that
\begin{equation}\label{BS-M}
 S_1(q)-qS_2(q)=\sum_{n=0}^{\infty}
 \frac{(q;q^2)_n(-q^2;q^4)_n(1-q^{2n+1})^2}
 {(q^4;q^4)_n(-q^3;q^6)_{n+1}}(-1)^nq^{n^2+2n}.
\end{equation}
Replacing $q$ by $q^2$ in \eqref{eq-Rahman46}, and then making the specialization
\begin{equation}\label{eq-cubic-1-proof-parameters}
       a=q^3,\quad b=-q^2,\quad d=\omega q, \quad \omega^3=1,\quad \omega \ne1,
\end{equation}
we obtain 
\begin{align}
 L(q) &:=\sum_{n=0}^{\infty}\frac{(q^3,\omega q,\omega^2q;q^2)_n(-q^2;q^4)_n}
 {(q^4;q^4)_n(\omega^2q^6,\omega q^6;q^4)_n(-q^3;q^2)_n}
 \frac{1-q^{6n+3}}{1-q^3}(-1)^nq^{n^2+2n} \nonumber \\
 &=\frac{(q^5,q^7,-\omega q^4,-\omega^2q^4;q^4)_\infty}{(-q^3,-q^5,\omega^2q^6,\omega q^6;q^4)_\infty}.\label{eq-cubic-1-L}
\end{align}
Note that
\begin{align}\label{eq-L-simplify}
 &\frac{(q^3,\omega q,\omega^2q;q^2)_n}{(\omega^2q^6,\omega q^6;q^4)_n(-q^3;q^2)_n}\frac{1-q^{6n+3}}{1-q^3} \nonumber \\
 &=\frac{(q^3;q^2)_n}{(-q^3;q^2)_n} \frac{(q^3;q^6)_n}{(q;q^2)_n}
 \frac{(q^6;q^4)_n}{(q^{18};q^{12})_n}\frac{1-q^{6n+3}}{1-q^3}=
 \frac{(q^3;q^2)_n^2}{(q;q^2)_n(-q^9;q^6)_n}.
\end{align}
Substituting it into \eqref{eq-cubic-1-L}, we deduce that
\begin{align}\label{eq-Lq-result}
 L(q)&=\sum_{n=0}^{\infty}
 \frac{(q;q^2)_n(-q^2;q^4)_n}
 {(q^4;q^4)_n(-q^9;q^6)_n}
 \left(\frac{1-q^{2n+1}}{1-q}\right)^2
 (-1)^nq^{n^2+2n}\nonumber \\
&=
 \frac{(q;q^2)_\infty^2(1+q^3)}{(1-q)^2}
 \frac{(-q^{12};q^{12})_\infty}
 {(-q^4;q^4)_\infty(q^{6};q^{12})_\infty} =\frac{1+q^3}{(1-q)^2}\frac{J_1^2J_4J_{24}}{J_2^2J_6J_8}.
\end{align}
Substituting \eqref{eq-Lq-result} into \eqref{BS-M}, we obtain
\begin{equation}\label{BS-minus}
 S_1(q)-qS_2(q)=\frac{J_1^2J_4J_{24}}{J_2^2J_6J_8}.
\end{equation}
From \eqref{BS-sum} and \eqref{BS-minus} we deduce that
\begin{align} 
2S_1(q)=\frac{J_3^2J_4J_{24}}{J_2J_6^2J_8}+\frac{J_1^2J_4J_{24}}{J_2^2J_6J_8},  \label{B-exp}\\
2qS_2(q)=\frac{J_3^2J_4J_{24}}{J_2J_6^2J_8}-\frac{J_1^2J_4J_{24}}{J_2^2J_6J_8}. \label{S-exp}
\end{align}
Using the Maple approach in \cite{Frye-Garvan}, we obtain \eqref{cubic-id-1} and \eqref{cubic-id-2} immediately. 

(2) We denote the the series on the left side of \eqref{cubic-id-3} and \eqref{cubic-id-4}by $S_3(q)$ and $S_4(q)$, respectively. Let
\begin{equation}\label{eq-cubic-thm1-R}
 R(q):=\sum_{n=0}^{\infty}
 \frac{(-1;q)_n(-1;q^2)_n}
 {(q;q^2)_n(q^3;q^3)_n}q^{n(n+3)/2}.
\end{equation}
Then it is easy to see that
\begin{equation}\label{eq-R-K}
 R(q)=1+4q^2S_3(q).
\end{equation}
Hence it suffices to evaluate $R(q)$. Note that 
\begin{equation}\label{eq-Rdef}
 R(q)=\sum_{n=0}^{\infty}A_n(q)\frac{2q^n}{1+q^{2n}}
\end{equation}
where
\begin{equation}\label{eq-An-def}
 A_n(q):=\frac{(-1;q)_n(-q^2;q^2)_n}
 {(q;q^2)_n(q^3;q^3)_n}q^{n(n+1)/2}.
\end{equation}
We further define
\begin{align}\label{eq-H-product}
    H(q):=\sum_{n=0}^{\infty}A_n(q)=
 \frac{(-q;q^2)_\infty(-q^3;q^6)_\infty}
 {(q;q^2)_\infty(q^3;q^6)_\infty}=\frac{J_2^3J_6^3}{J_1^3J_3^2J_4J_{12}}.
\end{align}
Here the penultimate equality follows from \eqref{id-cubic-C} with $a=-1$.

Let $\zeta=\e^{\pi i/3}$. Then $\zeta^3=-1$ and we have
\begin{align}
    (1+x)(1-\zeta x)(1-\zeta^{-1}x)=1+x^3.\label{zeta-id}
\end{align}
Hence 
\begin{align}
    (-a,\zeta a,\zeta^{-1} a;q)_n=(-a^3;q^3)_n, \quad n \in \mathbb{Z}_{\geq 0} \cup \{\infty\}.
\end{align}
We have
\begin{align}
    &H(q)+R(q)=\sum_{n=0}^\infty \frac{(-1;q)_n(-q^2;q^2)_n}{(q;q^2)_n(q^3;q^3)_n} \frac{(1+q^n)^2}{1+q^{2n}}q^{n(n+1)/2} \nonumber \\
    &=\sum_{n=0}^\infty \frac{(-1;q^2)_n(-q;q)_n(1+q^n)(-q^3;q^3)_n }{2(q;q^2)_n(q^6;q^6)_n}q^{n(n+1)/2} \nonumber \\
    &=\sum_{n=0}^\infty \frac{(-1;q^2)_n(-q;q)_n(1+q^n)(1+q^{3n})(-1;q^3)_n }{2(q;q^2)_n(q^2,-\zeta q^2,-q^2/\zeta;q^2)_n} q^{n(n+1)/2}\nonumber \\
    &=\sum_{n=0}^\infty \frac{(-1;q^2)_n(1+q^{3n})(\zeta,-q,\zeta^2;q)_n}{(q;q)_n(-q^2\zeta^2,q,-q^2\zeta;q^2)_n}q^{n(n+1)/2} \nonumber \\
    &=2 \frac{(-q^2;q^2)_\infty^2
       (\zeta q,\zeta^{-1}q;q^2)_\infty}
 {(q;q^2)_\infty^2
       (-\zeta^{-1}q^2,-\zeta q^2;q^2)_\infty} \nonumber \\
    &=2\frac{(-q^2;q^2)_\infty^2}{(q;q^2)_\infty^2} \frac{(-q^3;q^6)_\infty}{(-q;q^2)_\infty} \frac{(q^2;q^2)_\infty}{(q^6;q^6)_\infty}=2
\frac{J_4^3J_6}{J_1J_2J_3J_{12}}.  \label{eq-HplusR}
\end{align}
Here the fifth equality follows by setting  $a=-1$, $b=\zeta$ and  $c=-q$ in \eqref{eq-Rahman18}.

From \eqref{eq-HplusR} and \eqref{eq-H-product} we deduce that
\begin{align}\label{thm4-R-exp}
 R(q)=2\frac{J_4^3J_6}{J_1J_2J_3J_{12}}-\frac{J_2^3J_6^3}{J_1^2J_3^2J_4J_{12}}.
\end{align}
Substituting it into \eqref{eq-R-K} and then using the Maple approach in \cite{Frye-Garvan}, we obtain \eqref{cubic-id-3}.

Let $L_n$ denote the $n$th summand in the left side of \eqref{cubic-id-4}.  It is easy to see that for $n\geq 1$,
\begin{equation}\label{eq-Ln-An}
 L_n=A_n(q)\frac{1+q^n+q^{2n}}{2(1+q^{2n})}
 =A_n(q)\left(\frac12+\frac{q^n}{2(1+q^{2n})}\right).
\end{equation}
Since $A_0=1$, summing \eqref{eq-Ln-An} over $n\geq 1$ gives
\begin{align*}
 S_4(q)=\frac12\big(H(q)-1\big)
       +\frac14\big(R(q)-1\big).
\end{align*}
Substituting \eqref{eq-H-product} and \eqref{thm4-R-exp} into it, we obtain \eqref{cubic-id-4}.
\end{proof}

\begin{rem}
The identity \eqref{cubic-id-4}can also be proved in the following way. Replacing $n$ by $n+1$ in the summand, we have 
\begin{align}\label{cubic-thm4-L-L1L2}
    S_4(q)&=\sum_{n=0}^\infty \frac{(-q;q)_n(-q^2;q^2)_n}{(q;q^2)_{n+1}(q^3;q^3)_n(1-q^{n+1})}q^{(n+1)(n+2)/2} \nonumber \\
    &=\sum_{n=0}^\infty \frac{(1+q^{n+1}+q^{2n+2})(-q;q)_n(-q^2;q^2)_n}{(q;q^2)_{n+1}(q^3;q^3)_n(1-q^{3n+3})}q^{(n+1)(n+2)/2} \nonumber \\
    &=L_1(q)+q^2L_2(q)
\end{align}
where
\begin{align}
L_1(q)&=\sum_{n=0}^\infty \frac{(-q;q)_n(-q^2;q^2)_{n+1}}{(q;q^2)_{n+1}(q^3;q^3)_{n+1}}q^{(n+1)(n+2)/2}, \\
L_2(q)&=\sum_{n=0}^\infty \frac{(-q;q)_n(-q^2;q^2)_n}{(q;q^2)_{n+1}(q^3;q^3)_{n+1}}q^{(n^2+5n)/2}.
\end{align}
We have
\begin{align}
    L_1(q)&=\frac{1}{2}\Big(\sum_{n=0}^\infty \frac{(-1;q)_n(-q^2;q^2)_n}{(q;q^2)_n(q^3;q^3)_n} q^{n(n+1)/2}-1\Big)=\frac{1}{2}\Big(\frac{J_2^3J_6^3}{J_1^2J_3^2J_4J_{12}}-1  \Big),  \label{cubic-thm4-proof-L1}\\
    L_2(q)&=\frac{J_8J_{12}^5J_{2,24}}{J_4J_6J_{1,12}J_{2,12}J_{3,12}^2J_{5,12}}. \label{cubic-thm4-proof-L2}
\end{align}
Here the last equality in \eqref{cubic-thm4-proof-L1} and \eqref{cubic-thm4-proof-L2} follow from \eqref{eq-H-product} and \eqref{cubic-id-3}, respectively.
Substituting \eqref{cubic-thm4-proof-L1} and \eqref{cubic-thm4-proof-L2} into \eqref{cubic-thm4-L-L1L2}, we obtain 
\begin{align}\label{cubic-thm4-5-new}
    S_4(q)=\frac{1}{2}\Big(\frac{J_2^3J_6^3}{J_1^2J_3^2J_4J_{12}}-1\Big)+q^2\frac{J_8J_{12}^5J_{2,24}}{J_4J_6J_{1,12}J_{2,12}J_{3,12}^2J_{5,12}}.
\end{align}
This proves \eqref{cubic-id-4}upon using the Maple approach in \cite{Frye-Garvan} to show its equivalence with \eqref{cubic-thm4-5-new}.
\end{rem}

\begin{theorem}\label{thm-cubic-2}
Let $\omega$ be a primitive cubic root of unity. We have
\begin{align}
 \sum_{n=0}^\infty \frac{(-\omega;q)_n(-\omega^2;q^2)_n}{(\omega^2q;q^2)_n (q^3;q^3)_n}\omega^nq^{n(n+3)/2}
& = \frac{J_{5,12}}{J_3(\omega q^2;q^4)_\infty},  \label{eq-cubic-root-id-1}  \\
 \sum_{n=0}^{\infty}
 \frac{(\omega q^4;q^4)_n}
 {(q^3;q^3)_n(\omega q;q)_{n+1}(\omega^2q^;q^2)_{n+1}}
 \omega^nq^{n(n+3)/2}&=\frac{J_{1,12}}
 {J_3(\omega q^2;q^4)_\infty}.  \label{eq-cubic-root-id-2}
\end{align}
\end{theorem}
\begin{proof}
(1) We denote the left side of \eqref{eq-cubic-root-id-1} as
\begin{align}\label{eq-Udef}
 U(q):=\sum_{n=0}^{\infty}u_n(q),  \quad
 u_n(q):=\frac{(-\omega;q)_n(-\omega^2;q^2)_n}{(\omega^2q;q^2)_n(q^3;q^3)_n}\omega^nq^{n(n+3)/2}.
\end{align}

We define a closely related series which can be evaluated by \eqref{id-cubic-C}:
\begin{align}\label{eq-hAdef}
H(q):=\sum_{n=0}^{\infty}h_n(q), \quad 
 h_n(q):=\frac{(-\omega;q)_n(-\omega^2q^2;q^2)_n}{(\omega^2q;q^2)_n(q^3;q^3)_n}\omega^nq^{n(n+1)/2}.
\end{align}
Setting $a=-\omega$ in \eqref{id-cubic-C}, we obtain
\begin{align}\label{eq-Hsum}
 H(q)=\frac{(-\omega q;q^2)_\infty(-q^3;q^6)_\infty}{(\omega^2q;q^2)_\infty(q^3;q^6)_\infty}.
\end{align}
Note that
\begin{align}
 \frac{u_n(q)}{h_n(q)}
 =q^n\frac{(-\omega^2;q^2)_n}{(-\omega^2q^2;q^2)_n}
 =q^n\frac{1+\omega^2}{1+\omega^2q^{2n}}
 =-\frac{\omega q^n}{1+\omega^2q^{2n}}.
\end{align}
Consequently,
\begin{align}\label{eq-key-term}
&r_n(q):= 2u_n(q)-h_n(q)
 =-\frac{(1+\omega q^n)^2}{1+\omega^2q^{2n}}h_n(q) \nonumber \\
&=-\frac{(1+\omega q^n)^2}{1+\omega^2q^{2n}} \times \frac{(-\omega;q)_n(-\omega^2q^2;q^2)_n}
      {(\omega^2q;q^2)_n(q^3;q^3)_n}
 \omega^nq^{n(n+1)/2} \nonumber \\
&=-\frac{1+\omega}{1+\omega^2} \frac{(-\omega q;q)_n(-\omega^2;q^2)_n(1+\omega q^n)}{(\omega^2q;q^2)_n(q^3;q^3)_n}\omega^nq^{n(n+1)/2} \nonumber \\
&=-\omega \frac{(-\omega q;q)_n(-\omega^2;q^2)_n(-\omega^2q;q)_n(1+\omega q^n)}{(\omega^2q;q^2)_n(q,\omega q,\omega^2q;q)_n(-\omega^2q;q)_n}\omega^nq^{n(n+1)/2} \nonumber \\
&=-\frac{\omega}{1+\omega^2} \frac{(-q;q)_n(-\omega^2,-\omega q;q)_n(-\omega^2;q^2)_n(1+\omega q^n)(1+\omega^2q^n)}{(q^2;q^2)_n(\omega q;q)_n(\omega^2q,\omega q^2;q^2)_n} \omega^nq^{n(n+1)/2} \nonumber \\
&=\frac{(-1;q)_n(-\omega^2,-\omega q;q)_n(-\omega^2;q^2)_n(1+q^{3n})}{2(q^2;q^2)_n(\omega q;q)_n(\omega^2q,\omega q^2;q^2)_n} \omega^nq^{n(n+1)/2}.
\end{align}
Here we used the simple fact:
\begin{equation}\label{eq-cubic-factor-1}
\begin{split}
(1-x)(1-\omega x)(1-\omega^2x)&=1-x^3, \\
(x;q)_n(\omega x;q)_n(\omega^2x;q)_n&=(x^3;q^3)_n, \quad n \in \mathbb{Z}_{\geq 0} \cup\{\infty\}.
\end{split}
\end{equation}
Setting $a=-1$ and $b=d=-\omega^2$ in \eqref{eq-Rahman46}, then
\begin{align*}
 \frac{aq}{b}=\omega q, \quad  \frac{aq^2}{d}=\omega q^2, \quad
 adq=\omega^2q, \quad  \frac{aq^2}{bd}=-\omega^2q^2,
\end{align*}
and we deduce from \eqref{eq-Rahman46} that
\begin{align}\label{R-prod}
   R(q):=\sum_{n=0}^\infty r_n(q)=
 \frac{(-q;q^2)_\infty^2(-q^2;q^2)_\infty
       (-\omega^2q^2;q^2)_\infty}
 {(\omega q;q^2)_\infty(\omega q^2;q^2)_\infty^2
       (\omega^2q;q^2)_\infty}.
\end{align}

By definition we have
\begin{align}\label{U-H-R}
 U(q)=\frac{H(q)+R(q)}{2}.
\end{align}

Let
\begin{equation}
\label{eq-P}
 P(q):=\frac{J_{1,12}J_{5,12}^{2}}{J_{12}J_{3,12}J_6}
 \frac{(\omega q;q^2)_\infty}{(-\omega^2q;q^2)_\infty}.
\end{equation}
From \eqref{eq-cubic-factor-1} we have
\begin{align}
\label{eq-HX-EY}
 H(q)=P(q)X(q), \quad  E(q)=P(q)Y(q),
\end{align}
where
\begin{align}\label{eq-X}
 X(q)=\frac{J_{1,12}J_6^2}{J_{12}J_{2,12}J_{3,12}}, \quad 
 Y(q)=\frac{J_{2,12}^2J_{4,12}J_6^2}{J_{12}J_{1,12}J_{3,12}J_{5,12}^2}.
\end{align}
Thus the remaining assertion is the theta-product identity
\begin{align}\label{eq-XY}
 X(q)+Y(q)=2.
\end{align}
This can be proved automatically using the Maple approach in \cite{Frye-Garvan}.

(2) Let $v_n(q)$ denote the $n$th summand in the left side of \eqref{eq-cubic-root-id-2} and we use $V(q):=\sum_{n\geq 0}v_n(q)$ to denote the sum on its left side. We claim that for every $n\ge0$,
\begin{equation}
\label{eq-vu-ratio}
 \frac{v_n(q)}{u_n(q)}
 = \frac{(1+\omega q^n)(1+\omega^2q^{2n})}
 {(1-\omega q^{n+1})(1-\omega^2q^{2n+1})}.
\end{equation}
In fact, we first observe that
\begin{equation}
\label{eq-quartic-factorization}
 (\omega q^4;q^4)_n=(\omega q;q)_n(-\omega q;q)_n(-\omega^2q^2;q^2)_n.
\end{equation}
Using \eqref{eq-quartic-factorization} and cancelling the common factors
gives
\begin{align*}
\frac{v_n(q)}{u_n(q)}
&=\frac{(\omega q;q)_n}{(\omega q;q)_{n+1}}
\frac{(-\omega q;q)_n}{(-\omega;q)_n}
\frac{(-\omega^2q^2;q^2)_n}{(-\omega^2;q^2)_n}
\frac{(\omega^2q;q^2)_n}{(\omega^2q;q^2)_{n+1}}\\
&=\frac{1}{1-\omega q^{n+1}}
\frac{1+\omega q^n}{1+\omega}
\frac{1+\omega^2q^{2n}}{1+\omega^2}
\frac{1}{1-\omega^2q^{2n+1}}.
\end{align*}
Since $(1+\omega)(1+\omega^2)=1$, we prove \eqref{eq-vu-ratio}. 

From \eqref{eq-vu-ratio} we have
\begin{align}\label{v-u-u-ratio}
\frac{u_n(q)+qv_n(q)}{u_n(q)}
=1+q\frac{(1+\omega q^n)(1+\omega^2q^{2n})}{(1-\omega q^{n+1})(1-\omega^2q^{2n+1})}=\frac{(1+q)(1+q^{3n+1})}{(1-\omega q^{n+1})(1-\omega^2q^{2n+1})}.
\end{align}
Hence
\begin{align}
&\overline{r}_n(q):=u_n(q)+qv_n(q)=\frac{(1+q)(1+q^{3n+1})}{(1-\omega q^{n+1})(1-\omega^2q^{2n+1})}u_n(q) \nonumber \\
&=\frac{(1+q)(1+q^{3n+1})}{(1-\omega q^{n+1})(1-\omega^2q^{2n+1})} \times \frac{(-\omega;q)_n(-\omega^2;q^2)_n}{(\omega^2q;q^2)_n (q^3;q^3)_n}\omega^nq^{n(n+3)/2} \nonumber \\
&=\frac{1+q}{(1-\omega q)(1-\omega ^2q)}\times \frac{(\omega q;q)_n}{(\omega q^2;q)_n}\frac{(1+q^{3n+1})(-\omega ;q)_n(-\omega^2;q^2)_n}{(\omega^2q^3;q^2)_n(q^3;q^3)_n}\omega^nq^{n(n+3)/2} \nonumber \\
&=\frac{1+q}{(1-\omega q)(1-\omega ^2q)}\times \frac{(\omega q;q)_n}{(\omega q^2;q)_n}\frac{(1+q^{3n+1})(-\omega ;q)_n(-\omega^2;q^2)_n}{(\omega^2q^3;q^2)_n(q,\omega q,\omega^2q;q)_n}\omega^nq^{n(n+3)/2} \nonumber \\
&=\frac{1+q}{(1-\omega q)(1-\omega ^2q)}\times \frac{1}{(\omega q^2;q)_n}\frac{(1+q^{3n+1})(-\omega ;q)_n(-\omega^2;q^2)_n(-\omega^2q;q)_n}{(\omega^2q^3;q^2)_n(q;q)_n(\omega^4q^2;q^2)_n}\omega^nq^{n(n+3)/2} \nonumber \\
&=\frac{(1+q)^2}{(1-\omega q)(1-\omega ^2q)}\times\frac{(-q;q)_n(1+q^{3n+1})
      (-\omega,-\omega^2q;q)_n(-\omega^2;q^2)_n}
{(q^2;q^2)_n(1+q)
 (\omega^2q^3,\omega q^2;q^2)_n(\omega q^2;q)_n}
\omega^nq^{n(n+3)/2}.
\end{align}
Setting $a=-q$, $b=-\omega^2$ and $d=-\omega$ in \eqref{eq-Rahman46},  we deduce that
\begin{align}
\label{eq-bar-Rdef}
 \overline{R}(q)&:=\sum_{n=0}^{\infty} \overline{r}_n(q)
 =\frac{(1+q)^2}{(1-\omega q)(1-\omega ^2q)}\times
 \frac{(-q^2;q^2)_\infty(-q^3;q^2)_\infty^2
       (-\omega^2q^2;q^2)_\infty}
 {(\omega q^2;q^2)_\infty^2
  (\omega q^3;q^2)_\infty(\omega^2q^3;q^2)_\infty} \nonumber \\
  &=\frac{J_2^2J_6}{J_1J_3J_4}\frac{(-\omega^2q^2;q^2)_\infty}{(\omega q^2;q^2)_\infty^2}.
\end{align}
By definition we have
\begin{equation}
\label{eq-V-RU}
qV(q)=\overline{R}(q)-U(q).
\end{equation}

Now we denote the infinite product in the right side of \eqref{eq-cubic-root-id-2} as 
\begin{equation}
\label{eq-Pdef}
 \overline{P}(q):=
\frac{J_{1,12}}
 {J_3(\omega q^2;q^4)_\infty}.
\end{equation}
Then it is easy to see from \eqref{eq-cubic-root-id-1} that
\begin{align}\label{U-P-X}
     U(q)=\overline{P}(q)\overline{X}(q), \quad \overline{X}(q)=\frac{J_{5,12}}{J_{1,12}}. 
\end{align}
To compare $\overline{R}(q)$ with $\overline{P}(q)$, we note that
\begin{align}
   &\frac{(-\omega^2q^2;q^2)_\infty}{(\omega q^2;q^2)_\infty^2} (\omega q^2;q^4)_\infty =\frac{(-\omega^2q^2;q^2)_\infty(\omega^2q,-\omega^2q;q^2)_\infty}{(\omega q^2;q^2)_\infty^2} \nonumber \\
    &=\frac{(-\omega^2q;q)_\infty (\omega^2q;q^2)_\infty}{(\omega^4q^2;q^2)_\infty(\omega q^2;q^2)_\infty}=\frac{1}{(\omega ^2q^2;q^2)_\infty (\omega q^2;q^2)_\infty}=\frac{J_2}{J_6}.
\end{align}
Hence \eqref{eq-bar-Rdef} implies
\begin{equation}\label{R-P-Y}
\overline{R}(q)=\overline{P}(q)\overline{Y}(q), \quad \overline{Y}(q)= \frac{J_2^3}{J_1J_4J_{1,12}}.
\end{equation}
Using the Maple approach in \cite{Frye-Garvan}, we proved that
\begin{align}
    \overline{Y}(q)-\overline{X}(q)=q.
\end{align}
Substituting \eqref{U-P-X} and \eqref{R-P-Y}  into \eqref{eq-V-RU}, we conclude that $V(q)=\overline{P}(q)$. This proves \eqref{eq-cubic-root-id-2}.
\end{proof}

\section{Proofs of the quadruple sum identities}\label{sec-quadruple}
Throughout this section we let $\omega$ be a primitive cubic root of unity. We start with the integral form stated by the author  \cite[(4.18)]{Wang2023} of $F(u,v,w,t)$ defined in \eqref{F-quadruple-defn}:
\begin{align}
F(u,v,w,t)
&=\oint \sum_{i=0}^\infty \frac{(-uz)^i}{(q;q)_i} \sum_{j=0}^\infty \frac{(vz^2)^j}{(q^2;q^2)_j} \sum_{k=0}^\infty \frac{(-wz^3)^k}{(q^3;q^3)_k} \sum_{\ell=0}^\infty \frac{q^{2\ell^2-2\ell}(-tz^4)^\ell}{(q^4;q^4)_\ell} \nonumber \\
& \quad \quad  \times \sum_{s=-\infty}^\infty (-1)^s q^{\binom{s}{2}}z^{-s} \frac{dz}{2\pi \ii z} \nonumber \\
&=\oint \frac{(tz^4;q^4)_\infty (qz,1/z,q;q)_\infty}{(-uz;q)_\infty (vz^2;q^2)_\infty (-wz^3;q^3)_\infty} \frac{dz}{2\pi \ii z} \nonumber \\
&=\oint \frac{(t^{\frac{1}{4}}z,-t^{\frac{1}{4}}z,\ii t^{\frac{1}{4}}z,-\ii t^{\frac{1}{4}}z,qz,1/z,q;q)_\infty}{(-uz,v^{\frac{1}{2}}z,-v^{\frac{1}{2}}z,-w^{\frac{1}{3}}z,-\omega w^{\frac{1}{3}}z,-\omega^2 w^{\frac{1}{3}}z;q)_\infty} \frac{dz}{2\pi \ii z}. \label{QF-2}
\end{align}

For each of the quadruple sums in \eqref{I5-quadruple}--\eqref{I6-quadruple}, we have $t=v^2$. Since $(v^2z^4;q^4)_\infty=(vz^2,-vz^2;q^2)_\infty$, the penultimate line of \eqref{QF-2} reduces the quadruple sums in \eqref{I5-quadruple} and \eqref{I6-quadruple} to the triple sums in \eqref{I5} and \eqref{I6}, respectively. By contrast, the relation $t=v^2$ is not satisfied for the quadruple sums in \eqref{KR7}--\eqref{KR8a}. Therefore, these quadruple sums do not appear to admit analogous reductions to similar triple sums, which makes the corresponding identities considerably more difficult to prove.

\subsection{Proof of \eqref{KR7}}
We have 
\begin{align}
F(q,q^3,q^6,q^8)=\oint \frac{(q^2z,\ii q^2z,-\ii q^2z,qz,1/z,q;q)_\infty}{(-qz,q^{\frac{3}{2}}z,-q^{\frac{3}{2}}z,-\omega q^2z,-\omega^2 q^2z;q)_\infty} \frac{dz}{2\pi \ii z}.
\end{align}
Now applying Lemma \ref{lem-integral} with 
\begin{align}
A=4, B=1, C=5, D=0, a_1=q^2, a_2=\ii q^2, a_3=-\ii q^2, a_4=q, \nonumber \\
b_1=1, c_1=q^{\frac{3}{2}}, c_2=-q^{\frac{3}{2}}, c_3=-q, c_4=-\omega q^2, c_5=-\omega^2 q^2,
\end{align}
we deduce that
\begin{align}\label{quad-id1-F-split}
F(q,q^3,q^6,q^8)=T_1(q)+T_2(q)+T_3(q)+T_4(q)+T_5(q),
\end{align}
where
\begin{align}
    T_1(q)&=K_1(1;q)R_1(1;q), \quad T_2(q)=K_1(-1;q)R_1(-1;q), \label{quad-id1-proof-T1T2} \\
    T_4(q)&=K_2(\omega;q)R_2(\omega;q),\quad
 T_5(q)=K_2(\omega^2;q)R_2(\omega^2;q) \label{quad-id1-proof-T4T5}
\end{align}
and
\begin{align}
T_3(q)&=2\frac{(-q;q)_\infty^3 (q^4;q^4)_\infty}{(q^3;q^3)_\infty}\left(1+\frac{4}{3}\sum_{n=1}^\infty \frac{(-q;q)_{n-1}(-q^2;q^2)_{n-1}}{(q^3;q^3)_{n-1}(q;q^2)_n (1-q^n)}q^{n(n+1)/2}  \right).  \label{quad-id1-proof-T3} 
\end{align}
Here the relevant series are defined as
\begin{align}
K_1(\zeta;q)&=-\frac{1}{2}\frac{(\zeta q^{\frac{1}{2}};q)_\infty^3(-q;q^2)_\infty}{(-q;q)_\infty (-\zeta q^{\frac{3}{2}};q^3)_\infty}, \\
R_1(\zeta;q)&=\sum_{n=0}^\infty \frac{(\zeta q^{\frac{1}{2}};q)_n (-q;q^2)_n }{(q^2;q^2)_n (-\zeta q^{\frac{3}{2}};q^3)_n (1+\zeta q^{n+\frac{1}{2}})} (-\zeta )^nq^{\frac{1}{2}n^2+n}, \\
K_2(\zeta;q)&=
 \frac{(-\zeta q^2,-\zeta^2,-\zeta^2q^{-1};q)_\infty
       (-\zeta;q^2)_\infty}
 { (\zeta^2q^{-1},\zeta;q)_\infty
       (\zeta q^{-1};q^2)_\infty},     \label{eq-id1-K2}\\
R_2(\zeta;q)&=\sum_{n=0}^\infty
 \frac{(\zeta q^4;q^4)_n}
 {(q^3;q^3)_n(\zeta q^2;q)_n(\zeta^2q^3;q^2)_n}
 \zeta^nq^{n(n+3)/2}.       \label{eq-id1-R2}
\end{align}

 From \eqref{cubic-id-1} we have
\begin{align}\label{quad-id1-T1}
S(q):=T_1(q^2)=-\frac{1}{2}\frac{J_1^3J_3J_4J_{12}J_{24}^2J_{1,12}J_{3,12}}{J_2^3J_6^4J_8^2J_{2,24}}.
\end{align}
From \eqref{quad-id1-proof-T1T2} we see that
\begin{align}\label{quad-id1-T2}
T_2(q^2)=S(-q)=-\frac{1}{2}\frac{J_2^6J_{12}^4J_{6,24}}{J_1^3J_3J_4^2J_6J_8^2J_{1,12}J_{3,12}}.
\end{align}

Substituting \eqref{cubic-id-4}into \eqref{quad-id1-proof-T3}, we obtain
\begin{equation}\label{quad-id1-T3}
 T_3(q)=\frac{2J_2^6J_6^3}{3J_1^5J_3^3J_{12}}+\frac{4J_2^2J_4^4J_6}{3J_1^4J_3^2J_{12}}.
\end{equation}

Substituting \eqref{eq-cubic-root-id-2} into \eqref{quad-id1-proof-T4T5}, we deduce that
\begin{equation}\label{quad-id1-T4T5}
T_4(q)=T_5(q)=-\frac{q}{3}\frac{J_{1,12}^4J_{2,12}^2J_{4,12}J_{5,12}^3}
 {J_{3,12}J_6J_{12}^8}.
\end{equation}
Substituting \eqref{quad-id1-T1}--\eqref{quad-id1-T4T5} into \eqref{quad-id1-F-split}, and using the Maple approach in \cite{Frye-Garvan} to verify theta function identities, we obtain \eqref{KR7}.

\subsection{Proof of \eqref{KR7a}}
We have
\begin{align}
F(q^3,q^5,q^6,q^{12})=\oint \frac{(q^3z,\ii q^3z,-\ii q^3z,qz,1/z,q;q)_\infty}{(q^{\frac{5}{2}}z,-q^{\frac{5}{2}}z,-q^2z,-\omega q^2z,-\omega^2 q^2z;q)_\infty} \frac{dz}{2\pi \ii z}.
\end{align}
Now setting
\begin{align}
A=4, B=1, C=5, D=0, a_1=q^3, a_2=\ii q^3, a_3=-\ii q^3, a_4=q, \nonumber \\
b_1=1, c_1=q^{\frac{5}{2}}, c_2=-q^{\frac{5}{2}}, c_3=-q^2, c_4=-\omega q^2, c_5=-\omega^2 q^2
\end{align}
in Lemma \ref{lem-integral}, we deduce that
\begin{align}\label{id2-F-split}
F(q^3,q^5,q^6,q^{12})=T_1(q)+T_2(q)+T_3(q)+T_4(q)+T_5(q),
\end{align}
where
\begin{align}
    T_1(q)&=K_1(1;q)R_1(1;q), \quad  T_2(q)=K_1(-1;q)R_1(-1;q), \label{quad-id2-proof-T1T2} \\
    T_3(q)&=-2K_2(1;q)R_2(1;q), \quad T_4(q)=K_2(\omega;q)R_2(\omega;q), \label{quad-id2-proof-T3T4} \\
    T_5(q)&=K_2(\omega^2;q)R_2(\omega^2;q), \label{quad-id2-proof-T5}
\end{align}
with
\begin{align}
K_1(\zeta;q)&=\frac{1}{2}\frac{(\zeta q^{\frac{5}{2}},\zeta q^{\frac{1}{2}},\zeta q^{-\frac{3}{2}};q)_\infty (-q;q^2)_\infty}{(-q;q)_\infty (-\zeta q^{-\frac{3}{2}};q^3)_\infty}, \\
R_1(\zeta;q)&=\sum_{n=0}^\infty \frac{(\zeta q^{\frac{1}{2}};q)_n(-q;q^2)_n}{(q^2;q^2)_n(-\zeta q^{\frac{9}{2}};q^3)_n}(-\zeta)^nq^{\frac{1}{2}n^2+2n}, \\
K_2(\zeta;q)&=-\frac{1}{3}q^{-1}\frac{(-\zeta q^2;q^2)_\infty (-\zeta^2q;q)_\infty^2 (-\zeta q,q;q)_\infty}{(q^3;q^3)_\infty (\zeta q;q^2)_\infty}, \\
R_2(\zeta;q)&=\sum_{n=0}^\infty \frac{(-\zeta;q)_n (-\zeta^2;q^2)_n}{(q^3;q^3)_n (\zeta^2 q;q^2)_n} \zeta^n q^{\frac{1}{2}n^2+\frac{3}{2}n}.
\end{align}
From \eqref{cubic-id-2} we have
\begin{align}\label{quad-id2-T1}
   S(q):=T_1(q^2)=\frac{1}{2}q^{-1}\frac{J_1^3J_3J_4J_{3,12}J_{5,12}J_{24}^2}{J_2^3J_6^2J_8J_{6,12}J_{8,24}J_{10,24}}.
\end{align}
From \eqref{quad-id2-proof-T1T2} we have
\begin{align}\label{quad-id2-T2}
    T_2(q^2)=S(-q)=-\frac{1}{2}q^{-1}\frac{J_2^6J_6J_{24}^2\overline{J}_{3,12}\overline{J}_{5,12}}{J_1^3J_3J_4^2J_8J_{12}J_{6,12}J_{8,24}J_{10,24}}.
\end{align}

From \eqref{cubic-id-3} we have
\begin{align}\label{quad-id2-T3}
    T_3(q)=\frac{2}{3}q^{-1}\frac{J_2^3J_4}{J_1^3J_3}\Big(1+4q^2\frac{J_8J_{12}^5J_{2,24}}{J_4J_6J_{1,12}J_{2,12}J_{3,12}^2J_{5,12}}\Big).
\end{align}

Substituting \eqref{eq-cubic-root-id-1} into \eqref{quad-id2-proof-T3T4} and simplifying by \eqref{eq-cubic-factor-1}, we have
\begin{equation}\label{quad-id2-T4}
 T_4(q)=
 -\frac{J_1}{3q}
 \frac{J_{1,12}^{2}J_{2,12}J_{5,12}^{3}}
      {J_{12}^{3}J_{3,12}^{2}J_6^{2}}.
\end{equation}
Comparing \eqref{quad-id2-proof-T3T4} with \eqref{quad-id2-proof-T5}, we see that $T_5(q)$ can be obtained from $T_4(q)$ simply by replacing $\omega$ with $\omega^2$. Therefore, we conclude from \eqref{quad-id2-T4} that
\begin{equation}\label{quad-id2-T5}
T_5(q)=T_4(q).
\end{equation}
Substituting \eqref{quad-id2-T1}--\eqref{quad-id2-T5} into \eqref{id2-F-split}, we obtain \eqref{KR7a} upon using the Maple approach in \cite{Frye-Garvan} to verify theta function identities.

\subsection{Proof of \eqref{KR8}}

We have
\begin{align}
F(q^2,q^3,q^5,q^8)=\oint \frac{(q^2z,\ii q^2z,-\ii q^2z,qz,1/z,q;q)_\infty}{(q^{\frac{3}{2}}z,-q^{\frac{3}{2}}z,-q^{\frac{5}{3}}z,-\omega q^{\frac{5}{3}}z,-\omega^2q^{\frac{5}{3}}z;q)_\infty}\frac{dz}{2\pi \ii z}.
\end{align}
Now setting
\begin{align}
A=4, B=1, C=5, D=0, a_1=q^2, a_2=\ii q^2, a_3=-\ii q^2, a_4=q, \\
b_1=1, c_1=q^{\frac{3}{2}}, c_2=-q^{\frac{3}{2}}, c_3=-q^{\frac{5}{3}}, c_4=-\omega q^{\frac{5}{3}}, c_5=-\omega^2 q^{\frac{5}{3}}
\end{align}
in Lemma \ref{lem-integral}, we deduce that
\begin{align}
F(q^2,q^3,q^5,q^8)=T_1(q)+T_2(q)+T_3(q)+T_4(q)+T_5(q),
\end{align}
where
\begin{align}
    T_1(q)&=K_1(1;q)R_1(1;q), \quad T_2(q)=K_1(-1;q)R_1(-1;q), \label{quad-id3-T1T2-defn} \\
    T_3(q)&=K_2(1;q)R_2(1;q), T_4(q)=K_2(\omega;q)R_2(\omega;q), T_5(q)=K_2(\omega^2;q)R_2(\omega^2;q) \label{quad-id3-T3T4T5-defn}
\end{align}
with 
\begin{align}
  K_1(\zeta;q)&=-\frac{1}{2}\zeta q^{-\frac{1}{2}}\frac{(\zeta q^{\frac{1}{2}};q)_\infty^3 (-q;q^2)_\infty}{(-q;q)_\infty (-\zeta q^{\frac{1}{2}};q^3)_\infty}, \\
  R_1(\zeta;q)&=\sum_{n=0}^\infty \frac{(\zeta q^{\frac{1}{2}};q)_n (-q;q^2)_n}{(q^2;q^2)_n (-\zeta q^{\frac{5}{2}};q^3)_n}(-\zeta)^n q^{\frac{1}{2}n^2+n}, \\
  K_2(\zeta;q)&=-\frac{1}{3}\zeta q^{-\frac{1}{3}} \frac{(-\zeta^2 q^{\frac{1}{3}};q)_\infty^2 (-\zeta q^{\frac{2}{3}};q^2)_\infty (-\zeta q^{\frac{2}{3}},q;q)_\infty}{(\zeta q^{\frac{5}{3}};q^2)_\infty (q^3;q^3)_\infty}, \\
  R_2(\zeta;q)&=\sum_{n=0}^\infty \frac{(-\zeta q^{\frac{2}{3}};q)_n (-\zeta^2 q^{\frac{4}{3}};q^2)_n}{(q^3;q^3)_n (\zeta^2 q^{\frac{1}{3}};q^2)_{n+1}} \zeta^n q^{\frac{1}{2}n^2+\frac{7}{6}n}.  
  \end{align}
Replacing $q$ by $q^2$ and setting $a=-q^{-1}$ in \eqref{eq-Chu}, we obtain 
\begin{align}
    \sum_{n=0}^\infty \frac{(q;q^2)_n (-q^2;q^4)_n}{(q^4;q^4)_n(-q^5;q^6)_n}(-1)^n q^{n^2+2n} =\frac{(-q^2;q^4)_\infty (q^3;q^6)_\infty (q^{12};q^{12})_\infty}{(q^4;q^4)_\infty (-q^5;q^6)_\infty (-q^2;q^{12})_\infty}. \label{eq-2358-a}
\end{align}
By \eqref{eq-2358-a} and \eqref{quad-id3-T1T2-defn}, we have
\begin{align}
T_1(q^2)=-\frac{1}{2}q^{-1}\frac{(-q^2;q^4)_\infty (q^6;q^6)_\infty}{(q^2;q^4)_\infty (q^8;q^8)_\infty (-q^2;q^{12})_\infty} (q;q^2)_\infty^4.
\end{align}
If we denote $S_1(q)=T_1(q^2)$. then we have $T_2(q^2)=S_1(-q)$. Hence
\begin{align}\label{T1plusT2}
&T_1(q)+T_2(q)=2S_1(q)|U_2 \nonumber \\
&=-\big(q^{-1}(q;q^2)_\infty^4\big)|U_2 \times \frac{(-q;q^2)_\infty (q^3;q^3)_\infty}{(q;q^2)_\infty (q^4;q^4)_\infty (-q;q^6)_\infty} \nonumber \\
&=4\frac{J_2J_3J_4^2}{J_1^4(-q;q^6)_\infty}.  \quad \text{(by \eqref{J2J1-U2})}
\end{align}

Replacing $q$ by $q^3$ and setting $a=-q^2$ in \eqref{id-cubic-C}, we obtain 
\begin{align}
\sum_{n=0}^\infty \frac{(-q^2;q^3)_n (-q^4;q^6)_n}{(q^9;q^9)_n(q;q^6)_{n+1}}q^{(3n^2+7n)/2}=\frac{(-q^5;q^6)_\infty (-q^{15};q^{18})_\infty}{(q;q^6)_\infty (q^9;q^{18})_\infty}. \label{eq-2358-b}
\end{align}

If we denote $S_2(q)=T_3(q^3)$, then by \eqref{eq-2358-b} and \eqref{quad-id3-T3T4T5-defn} we have
\begin{align}
S_2(q)&=-\frac{1}{3}q^{-1}\frac{(-q;q^3)_\infty^2(-q^2;q^6)_\infty(-q^2;q^3)_\infty (q^3;q^3)_\infty}{(q^5;q^6)_\infty (q^9;q^9)_\infty} \times \frac{(-q^5;q^6)_\infty (-q^{15};q^{18})_\infty}{(q;q^6)_\infty (q^9;q^{18})_\infty} \nonumber \\
&=-\frac{1}{3}q^{-1}(-q;q)_\infty^3 (-q^{15};q^{18})_\infty \frac{J_3^4J_{18}}{J_6^3J_9^2}.
\end{align}
It is easy to see that $T_4(q^3)=S_2(\omega^2 q)$ and $T_5(q^3)=S_2(\omega q)$. Hence
\begin{align}
&T_3(q)+T_4(q)+T_5(q)=3S_2(q)\mid U_3 \nonumber \\
&=-\big(q^{-1}(-q;q)_\infty^3\big)|U_3 \times (-q^5;q^6)_\infty \frac{J_6J_1^3}{J_2^3J_3^2}
&=-3(-q^5;q^6)_\infty \frac{J_2J_3^3}{J_1^4}.\quad \text{(by \eqref{J2J1-U3})} \label{2358-T345}
\end{align}
Adding \eqref{T1plusT2} and \eqref{2358-T345} together, we see that \eqref{KR8} is equivalent to
\begin{align}
4\frac{J_2J_3J_4^2}{J_1^4(-q;q^6)_\infty}-3(-q^5;q^6)_\infty \frac{J_2J_3^3}{J_1^4}=\frac{1}{(q^2,q^3,q^4,q^5,q^8,q^9,q^{11};q^{12})_\infty}.
\end{align}
Multiplying both sides by $(q^2,q^3,q^5;q^6)_\infty (q^4;q^{12})_\infty$, we see that it is equivalent to
\begin{align}
4\frac{J_3J_4^3}{J_1^3J_{12}}-3\frac{J_2^2J_3^4}{J_1^4J_6^2}=1. \label{J-eta-id}
\end{align}
This appears in \cite[(22.7.4)]{Hirschhorn}. This proves \eqref{KR8}.

\subsection{Proof of \eqref{KR8a}}
We denote the left side of \eqref{KR8a} as $L(q)$ and let
\begin{equation}\label{eq-Fr}
 F_r:=F(q^{r+2},q^{2r+3},q^{3r+4},q^{4r+8}).
\end{equation}
Then
\begin{align}\label{mix-L-split}
L(q)=F_0+qF_1+q^3F_2-q^7F_3.
\end{align}
We have
\begin{align}
    F_r&=\sum_{i,j,k,\ell \geq 0}\frac{(-1)^{\ell}q^{\binom{i+2j+3k+4\ell}{2}+2\ell^2+2i+3j+4k+6\ell+r(i+2j+3k+4\ell)}}{(q;q)_{i}(q^2;q^2)_{j}(q^3;q^3)_{k}(q^4;q^4)_{\ell}} \nonumber \\
    &=q^{-\binom{r}{2}}\sum_{i,j,k,\ell \geq 0}\frac{(-1)^{\ell}q^{\binom{i+2j+3k+4\ell+r}{2}+2\ell^2+2i+3j+4k+6\ell}}{(q;q)_{i}(q^2;q^2)_{j}(q^3;q^3)_{k}(q^4;q^4)_{\ell}}.
\end{align}
We define
\begin{align}\label{eq-Bp}
I(z)&:=\sum_{i,j,k,\ell\geq 0}
\frac{(-1)^{i+k+\ell}
 q^{2i+3j+4k+2\ell^2+6\ell}
 z^{i+2j+3k+4\ell}}
{(q;q)_i(q^2;q^2)_j(q^3;q^3)_k(q^4;q^4)_\ell} \nonumber \\
 &=\frac{(q^8z^4;q^4)_\infty}
{(-q^2z;q)_\infty(q^3z^2;q^2)_\infty
 (-q^4z^3;q^3)_\infty}.
\end{align}

It follows that for $r=0,1,2,3$,
\begin{equation}\label{eq-FrCT}
 F_r=(-1)^rq^{-\binom{r}{2}}\mathrm{CT}_z\left[z^rI(z)(qz,1/z,q;q)_\infty\right].
\end{equation}

Let $p(z)=1-qz+q^2z^2+q^4z^3$.  
From \eqref{mix-L-split} and \eqref{eq-FrCT} we have
\begin{align}\label{mix-L-integral}
L(q)&=\mathrm{CT}_z\left[p(z)I(z)(qz,1/z,q;q)_\infty \right] \nonumber \\
&=\oint \frac{p(z)(qz,q^2z,\ii q^2 z,-\ii q^2z,1/z,q;q)_\infty}{(q^{3/2}z,-q^{3/2}z,-q^{4/3}z,-\omega q^{4/3}z,-\omega^2 q^{4/3}z;q)_\infty} \frac{dz}{2\pi \ii z}.
\end{align}
Now set
\begin{align}
&A=4, B=1, C=5,D=0, a_1=q^2,  a_2=\ii q^2, 
 a_3=-\ii q^2,  a_4=q, \nonumber \\
 &b_1=1, c_1=q^{3/2}, c_2=-q^{3/2},
 c_3=-q^{4/3},  c_4=-\omega q^{4/3}, 
 c_5=-\omega^2q^{4/3}
\label{eq-parameters}
\end{align}
in Lemma~\ref{lem-integral}, we deduce that
\begin{align}\label{quad-id4-L-split}
    L(q)=T_1(q)+T_2(q)+T_3(q)+T_4(q)+T_5(q)
\end{align}
where 
\begin{align}
    T_1(q)&=K_1(1;q)R_1(1;q),
 \quad T_2(q)=K_{1}(-1;q)R_1(-1;q),  \label{quad-id4-proof-T1T2}\\
T_3(q)&=K_2(1;q)R_2(1;q),
T_4(q)=K_2(\omega;q)R_2(\omega;q),
T_5(q)=K_2(\omega^2;q)R_2(\omega^2;q), \label{quad-id4-proof-T3T4T5}
\end{align}
with
\begin{align} 
K_1(\zeta;q)&:=\frac{(\zeta q^{3/2},\zeta q^{-1/2}, \zeta q^{1/2};q)_\infty(-q;q^2)_\infty}
{(-1;q)_\infty(-\zeta q^{-1/2};q^3)_\infty}, \label{id4-K1}\\
R_1(\zeta;q)&:=\sum_{n=0}^\infty \frac{(\zeta q^{1/2};q)_n(-q;q^2)_n}
{(q^2;q^2)_n(-\zeta q^{7/2};q^3)_n}(-\zeta)^nq^{(n^2+4n)/2} \nonumber \\
&\quad \times \left(1-\zeta q^{-n-\frac{1}{2}}+q^{-2n-1} +\zeta q^{-3n-\frac{1}{2}}\right),
\label{id4-R1} \\
K_2(\zeta;q)&:=\frac{(-\zeta q^{4/3},-\zeta^2q^{-1/3},-\zeta^2q^{2/3};q)_\infty
      (-\zeta q^{4/3};q^2)_\infty}
{(\omega,\omega^2;q)_\infty (\zeta q^{1/3};q^2)_\infty},  \label{id4-K2}\\
R_2(\zeta;q)&:=\sum_{n=0}^\infty \frac{(-\zeta q^{1/3};q)_n (-\zeta^2q^{2/3};q^2)_n}
{(q^3;q^3)_n(\zeta^2q^{5/3};q^2)_n}\zeta^nq^{(3n^2+11n)/6}
\notag\\
&\quad\times \left(1+\zeta^2q^{-n-\frac{1}{3}} +\zeta q^{-2n-\frac{2}{3}}-q^{-3n}\right).
\label{eq-id4-R2}
\end{align}
If we denote $S_1(q)=T_1(q^2)$ and $S_2(q)=T_3(q^3)$, then from \eqref{quad-id4-proof-T1T2}--\eqref{quad-id4-proof-T3T4T5} we have
\begin{align}
 T_2(q^2)=S_1(-q), \quad   T_4(q)=S_2(\omega q),\quad T_5(q)=S_2(\omega^2 q).
\end{align}
We now evaluate $S_1(q)$ and $S_2(q)$. Note that
\begin{align}
1-q^{-n-\frac{1}{2}}+q^{-2n-1}+q^{-3n-\frac{1}{2}}=q^{-3n-\frac{1}{2}}(1+q^{3n+\frac{1}{2}})+q^{-2n-1}(1-q^{n+\frac{1}{2}}).
\end{align}
We have
\begin{align}\label{eq-R1-new}
R_1(1;q)&=\sum_{n=0}^\infty \frac{(q^{1/2};q)_n(-q;q^2)_n(-1)^nq^{\frac{1}{2}(n^2-1)-n}}{(q^2;q^2)_n(-q^{7/2};q^3)_{n-1}} \nonumber \\
&\quad \quad +\sum_{n=0}^\infty \frac{(q^{\frac{1}{2}};q)_n(-q;q^2)_n(-1)^nq^{\frac{1}{2}n^2-1}}{(q^2;q^2)_n(-q^{7/2};q^3)_n}(1-q^{n+\frac{1}{2}}) \nonumber \\
&=-\frac{1}{1+q^{1/2}}\sum_{n=-1}^\infty \frac{(-1)^nq^{n^2/2+2n}(q^{3/2};q)_n(-q;q^2)_n}
{(q^4;q^2)_n(-q^{7/2};q^3)_n}.
\end{align}
Here we replaced $n$ by $n+1$ in the first sum and then combine the two sums together.   Now we replace $q$ by $q^2$.  We write $n=m-1$ with $m\geq 0$.  For any $m\geq 0$ we have
\begin{align*}
(q^3;q^2)_{m-1}&=\frac{(q;q^2)_m}{1-q},
\quad
(-q^2;q^4)_{m-1}=\frac{(-q^2;q^4)_m}{1+q^{4m-2}},\\
(q^8;q^4)_{m-1}&=\frac{(q^4;q^4)_m}{1-q^4},
\quad
(-q^7;q^6)_{m-1}=\frac{(-q;q^6)_m}{1+q}.
\end{align*}
Hence
\begin{align}
R_1(1;q^2)&=q^{-3}(1+q)(1+q^2)\sum_{m\geq 0}\frac{(q;q^2)_m(-q^2;q^4)_m}{(q^4;q^4)_m(-q;q^6)_m(1+q^{4m-2})}(-1)^m q^{m^2+2m} \nonumber \\
&=q^{-1}(1+q)\sum_{m\geq 0}\frac{(q;q^2)_m(-q^{-2};q^4)_m}{(q^4;q^4)_m(-q;q^6)_m}
(-1)^m q^{m^2+2m} \nonumber \\
&=q^{-1}(1+q)\frac{(-q^2;q^4)_\infty(q^3;q^6)_\infty(q^{12};q^{12})_\infty}{(q^4;q^4)_\infty(-q;q^6)_\infty (-q^{10};q^{12})_\infty}.  \label{quad-id4-R1-product}
\end{align}
Here the last equality follows from \eqref{eq-Chu} with $q$ replaced by $q^2$ and $a=-q^{-5}$. 

From  \eqref{quad-id4-R1-product} we have
\begin{align}
    S_1(q)=T_1(q^2)=-\frac{1}{2}q^{-1}\frac{(q;q^2)_\infty^4 (q^4;q^{24})_\infty (q^{12};q^{24})_\infty^2}{(q^2;q^{12})_\infty^2 (q^6,q^{10};q^{12})_\infty (q^8,q^{16};q^{24})_\infty}.
\end{align}
It follows that
\begin{align}\label{quad-id4-T1T2-sum}
   & T_1(q)+T_2(q)=S_1(q)+S_1(-q)=2S_1(q)|U_2 \nonumber \\
   &=4\frac{J_4^{4}}{J_1^2J_2^2}\frac{(q^2;q^{12})_\infty (q^6;q^{12})_\infty^2}{(q;q^6)_\infty^2 (q^3,q^5;q^6)_\infty (q^4,q^8;q^{12})_\infty}
\end{align}
Here for the second equality we used \eqref{J2J1-U2}.

Note that
\begin{align}
1+q^{-n-\frac{1}{3}}+q^{-2n-\frac{2}{3}}-q^{-3n}=-q^{-3n}(1-q^{3n})+q^{-2n-\frac{2}{3}}(1+q^{n+\frac{1}{3}}).
\end{align}
We have
\begin{align}
    R_2(1;q)&=-\sum_{n=0}^\infty \frac{(-q^{1/3};q)_n(-q^{2/3};q^2)_nq^{(3n^2-7n)/6}}{(q^3;q^3)_{n-1}(q^{5/3};q^2)_n} \nonumber \\
&\quad +\sum_{n=0}^\infty \frac{(-q^{1/3};q)_{n+1}(-q^{2/3};q^2)_nq^{(3n^2-n-4)/6}}{(q^3;q^3)_n(q^{5/3};q^2)_n} \nonumber \\
&=-\frac{(1+q)(1+q^{1/3})}{1-q^{5/3}}\sum_{n\geq 0}\frac{(-q^{4/3};q)_n(-q^{2/3};q^2)_n}{(q^3;q^3)_n(q^{11/3};q^2)_n}q^{(3n^2+11n)/6}.
\label{eq-def-R}
\end{align}
Here for the last equality we replaced $n$ by $n+1$ in the first sum and then combine the two sums together. 
Setting $a=-q^{4/3}$ in \eqref{id-cubic-C}, we obtain 
\begin{align}
\sum_{n\geq 0}\frac{(-q^{4/3};q)_n(-q^{2/3};q^2)_n}{(q^3;q^3)_n(q^{11/3};q^2)_n}
q^{(3n^2+11n)/6}=\frac{(-q^{7/3};q^2)_\infty(-q^7;q^6)_\infty}{(q^{11/3};q^2)_\infty(q^3;q^6)_\infty}.
\label{eq-R-mid}
\end{align}
Substituting \eqref{eq-R-mid} into \eqref{eq-def-R}, we deduce that
\begin{equation}\label{quad-id4-R13-product}
R_2(1;q)=\frac{(-q^{1/3};q^2)_\infty (-q;q^6)_\infty}{(q^{5/3};q^2)_\infty (q^3;q^6)_\infty}.
\end{equation}

Hence we have
\begin{align}
    S_2(q)=T_3(q^3)=-\frac{1}{3}q^{-1}\frac{(q^2;q^2)_\infty^3}{(q;q)_\infty^3} \frac{(q^3;q^6)_\infty (q^6;q^9)_\infty (q^{12};q^{18})_\infty (q^6;q^{36})_\infty}{(-q^3;q^3)_\infty^2 (q^9;q^{18})_\infty}.
\end{align}
Hence we have
\begin{align}\label{quad-id4-T3T4T5-sum}
&T_3(q)+T_4(q)+T_5(q)=S_2(q)+S_2(\omega q)+S_2(\omega^2 q)=3S_2(q)|U_3  \nonumber \\
&=-3\frac{J_2^4J_3^5}{J_1^8J_6}\frac{(q;q^2)_\infty (q^2;q^3)_\infty (q^4,q^6)_\infty (q^2;q^{12})_\infty}{(-q;q)_\infty^2 (q^3;q^6)_\infty}
\end{align}
Here for the last equality we used \eqref{J2J1-U3}. 

Substituting \eqref{quad-id4-T1T2-sum} and \eqref{quad-id4-T3T4T5-sum} into \eqref{quad-id4-L-split} we obtain a product representation for $L(q)$. After rearranging the terms, we see that \eqref{KR8a} is equivalent to the following theta function identity
\begin{align}
    4\frac{J_4^4J_{2,12}J_6^2}{J_1^3J_2J_{4,12}J_{12}^2}-3\frac{J_2J_3^3J_{2,12}}{J_1^4J_{12}} =\frac{J_{12}^2}{J_{3,12}J_{4,12}}.
\end{align}
This can be proved easily using the Maple approach in \cite{Frye-Garvan}. This proves \eqref{KR8a}.

\subsection*{Acknowledgements}
This work was supported by the National Key R\&D Program of China (Grant No.\ 2024YFA1014500). The author is grateful to Prof.~ Michael Schlosser for helpful correspondence in 2022 on some cubic identities. 

\subsection*{Declaration of AI use}
ChatGPT (OpenAI, GPT-5.6 Sol) was used to assist with calculations and to develop several key steps to complete the proofs in Section \ref{sec-cubic}. It was also used to improve the language and presentation of the manuscript. All AI-assisted calculations, arguments, and suggestions were independently verified and revised by the author, who takes full responsibility for the content of the article.

\end{document}